\documentclass[12pt, reqno]{amsart}
\usepackage{amssymb,dsfont}
\usepackage{eucal}
\usepackage{amsmath}
\usepackage{amscd}
\usepackage[dvips]{color}
\usepackage{multicol}
\usepackage[all]{xy}           
\usepackage{graphicx}
\usepackage{color}
\usepackage{colordvi}
\usepackage{xspace}
\usepackage{txfonts}
\usepackage{lscape}
\usepackage{tikz} 
\numberwithin{equation}{section}
\usepackage[shortlabels]{enumitem}
\usepackage{ifpdf}
\ifpdf
\usepackage[colorlinks,final,backref=page,hyperindex]{hyperref}
\else
\usepackage[colorlinks,final,backref=page,hyperindex]{hyperref}
\fi
\usepackage{tikz}
\usepackage[active]{srcltx}
\usepackage{array}
\usepackage{tabularx}
\usepackage{colortbl}

\makeatletter

\theoremstyle{plain}

\numberwithin{equation}{section}

\newtheorem{theo}{Theorem}[section]
\newtheorem{pro}[theo]{Proposition}
\newtheorem{lem}[theo]{Lemma}
\newtheorem{cor}[theo]{Corollary}
\newtheorem{rem}[theo]{Remark}
\theoremstyle{definition}
\newtheorem{defi}[theo]{Definition}
\newtheorem{exa}[theo]{Example}

\def\ot{\otimes}

\def\la{\lambda}

\def\bC{{\mathbb C}}

\def\1{{\bf 1}}

\begin{document}
	\title[The N=1 BMS superalgebra]
	{Singular vectors, characters, and  composition series for the N=1 BMS superalgebra}

	\author[Jiang]{Wei Jiang}
	\address{Department of Mathematics, Changshu Institute of Technology, Changshu 215500, China}\email{jiangwei@cslg.edu.cn}
	
	\author[Liu]{Dong Liu}
	\address{Department of Mathematics, Huzhou University, Zhejiang Huzhou, 313000, China}\email{liudong@zjhu.edu.cn}

	\author[Pei]{Yufeng Pei}
	\address{Department of Mathematics, Huzhou University, Zhejiang Huzhou, 313000, China}\email{pei@zjhu.edu.cn}
	
	\author[Zhao]{Kaiming Zhao}
	\address{Department of Mathematics, Wilfrid Laurier University, Waterloo, ON, Canada N2L3C5, and School of
		Mathematical Science, Hebei Normal (Teachers) University, Shijiazhuang, Hebei, 050024 P. R. China}\email{kzhao@wlu.ca}
	
	\thanks{}

	\begin{abstract}
		This paper investigates the structure of Verma modules over the N=1 BMS superalgebra. 
		We provide a detailed classification of singular vectors, establish necessary and sufficient conditions for the existence of subsingular vectors, uncover the structure of maximal submodules,  present the composition series of Verma modules,  and derive character formulas for irreducible highest weight modules.
	\end{abstract}
	
	\maketitle

	\tableofcontents
	
	\section{Introduction}

	Infinite-dimensional symmetries play a significant role in physics. Specifically, Virasoro-type symmetries have  significant applications in two-dimensional field theory, string theory, gravity, and other areas. The representation theory of the Virasoro algebra has also been widely and deeply studied \cite{FF, IK, MY,RW}. In recent years, two-dimensional non-relativistic conformal symmetries have gained importance in establishing holographic dualities beyond the AdS/CFT correspondence \cite{Ba0,BT,SZ}.  The Bondi-Metzner-Sachs  algebra, commonly known as BMS algebra, generates the asymptotic symmetry group of three-dimensional Einstein gravity \cite{BM,BH,Sa}. Although BMS algebra extends the Virasoro algebra, its representation theory differs fundamentally. Studies on  special highest weight modules for the BMS algebra have explored various aspects: determinant formulas \cite{BGMM}, character formulas \cite{O}, free field realizations \cite{BJMN}, and modular invariance \cite{BSZ, BNSZ}. However, a complete understanding of   highest weight modules is still lacking.

	In mathematical literature, the BMS algebra is   known as the Lie algebra $W(2,2)$, an infinite-dimensional Lie algebra first introduced in \cite{ZD} to study the classification of moonshine-type vertex operator algebras generated by two weight-2 vectors. They examined the vacuum modules of the $W(2,2)$ algebra (with a VOA structure) and established necessary and sufficient conditions for these modules to be irreducible. Their key insight was creating a total ordering on the PBW bases, which facilitated   computations of determinant formulas (see also \cite{JPZ}). It is worth mentioning that the $W(2,2)$ algebra has also been discovered and studied in several different mathematical fields, such as \cite{FK, HSSU, Wi}. The irreducibility conditions for Verma modules over the $W(2,2)$ algebra are also given in \cite{Wi}. In \cite{JP}, it was proposed that maximal submodules of reducible Verma modules are generated by singular vectors. However, Radobolja \cite{R} pointed out that this is true only for typical highest weights. For atypical weights, the maximal submodules are generated by both a singular vector and a subsingular vector. He also derived a character formula for irreducible highest weight modules and established necessary conditions for subsingular vector existence. The study further conjectured that these necessary conditions are also sufficient. Later, \cite{JZ} provided additional support for this conjecture.
	Adamovic et al.   used the free field realization of the twisted Heisenberg-Virasoro algebra at level zero \cite{ACKP,Bi}, along with constructing screening operators in lattice vertex algebras, to derive an expression for singular vectors of Verma modules for the $W(2,2)$ algebra under certain conditions in  \cite{AR1,AR2}. To our knowledge, explicit formulas for singular and subsingular vectors, as well as the composition series for general Verma modules over the $W(2,2)$ algebra, remain unresolved prior to the present paper.

	The {N}=1 BMS superalgebra, introduced in \cite{BDMT}, is the minimal supersymmetric extension of the BMS$_3$ algebra with central extensions. It incorporates a set of spin-$\frac{3}{2}$ generators $ Q_n $ within the BMS$_3$ algebra framework. Although this superalgebra is a subalgebra of the truncated Neveu-Schwarz superalgebra, its representation theory differs significantly from that of the {N}=1 Neveu-Schwarz superalgebra \cite{BMRW,IK0,IK1,MR}.
	In recent paper \cite{LPXZ, DGL}, the authors classified simple smooth modules including Whittaker modules over the N=1 BMS superalgebra under mild conditions and provided necessary and sufficient conditions for the irreducibility of Verma modules and Fock modules. Further detailed analysis on the structure of reducible Verma modules over the {N}=1  BMS superalgebra will be carried out in the present paper.

	As established in \cite{LPXZ}, the Verma module $V(c_L,c_M,h_L,h_M)$ over the N=1 BMS superalgebra $\frak g$ is irreducible if and only if  $2h_M+\frac{p^2-1}{12}c_M\ne 0$ for any positive integer $p$. If further $c_M=0$, then $h_M=0$, resulting in the degeneration of the irreducible highest weight module into an irreducible highest weight module over the Virasoro algebra (refer to Lemma \ref{degenerated-case}).  In this paper, we  study the structure of the Verma module $V(c_L,c_M,h_L,h_M)$ over $\frak g$ under the obvious  conditions that 
	$$
	c_M\ne 0\  \text{and}\ \ 2h_M+\frac{p^2-1}{12}c_M=0\  \text{for some}\   p\in\mathbb Z_+.
	$$	
	We classify all singular vectors of the Verma module  $V(c_L,c_M,h_L,h_M)$  and provide explicit formulas. We also identify the necessary and sufficient conditions for  the existence of subsingular vector and list them all.  Our first main result is as follows:
	
	\vskip 0.2cm
	
	\noindent {\bf Main Theorem 1.}  (Theorems \ref{main1}, \ref{main2}, \ref{main3} below)  {\it Let $(c_L,c_M,h_L,h_M)\in\bC^4$ such that $2h_M+\frac{p^2-1}{12}c_M=0$ for some $p\in\mathbb Z_+$ with $c_M\ne0$.
		
		$ (1)$ All singular vectors in $V(c_L,c_M,h_L,h_M)$ are of the form ${\rm S}^i\1$ (when $p$ even) or ${\rm R}^i\1$ (when $p$ odd) for $ i\in \mathbb N$, where  ${\rm S}$ and ${\rm R}$ are given in Proposition \ref{singular-S1} and Proposition \ref{singular-R11}, respectively.
		
		$(2)$ There exists a subsingular vector of  $V(c_L,c_M,h_L,h_M)$ if and only if  $h_L=h_{p,r}$ for some $r\in\mathbb Z_+$, where
		\begin{eqnarray*}\label{e3.37}\label{atypical}
			h_{p,r}=-\frac{p^2-1}{24}c_L+\frac{(41p+5)(p-1)}{48}+\frac{(1-r)p}{2}-\frac{1+(-1)^p}8p.
		\end{eqnarray*}
		In this case, ${\rm T}_{p, r}\1$ is the unique subsingular vector, up to a scalar multiple, where ${\rm T}_{p, r}$ are given in Theorem \ref{main3}.
		
	}

	\vskip 0.2cm

	By utilizing the information provided in Main Theorem 1 regarding singular and subsingular vectors, we can derive the character formulas for irreducible highest weight modules over $\mathfrak{g}$ as follows:
	
	\vskip 0.2cm
	
	\noindent {\bf Main Theorem 2.}  (Theorems \ref{irreducibility}, \ref{irreducibility1} below)	{\it
		Let $(c_L,c_M,h_L,h_M)\in\bC^4$ such that $2h_M+\frac{p^2-1}{12}c_M=0$ for some $p\in\mathbb Z_+$ with $c_M\ne0$.
		
		$ (1)$  If $V(c_L,c_M,h_L,h_M)$ is typical (i.e., $h_L\ne h_{p,r}$ for any $r\in\mathbb Z_+$), then the maximal submodule  of $V(c_L,c_M,h_L,h_M)$ is generated by ${\rm S}\1$ (when $p$ is even), or by ${\rm R}\1$ (when $p$ is  odd). Additionally, the character formula of the irreducible highest weight module $L(c_L,c_M,h_L,h_M)$ can be expressed as follows:
		$$
		{\rm char}\, L(c_L,c_M,h_L,h_M)= q^{h_L}(1-q^{\frac{p}2})\left(1+\frac12(1+(-1)^p)q^{\frac p2}\right)\prod_{k=1}^{\infty}\frac{1+q^{k-\frac{1}{2}}}{(1-q^{k})^{2}}.		$$
		
		$ (2)$ If $V(c_L,c_M,h_L,h_M)$ is atypical (i.e., $h_L=h_{p,r}$ for some $r\in\mathbb Z_+$), then the maximal submodule of $V(c_L,c_M,h_L,h_M)$ is generated by the subsingular vector ${\rm T}_{p,r}\1$.  Additionally, the character formula of the irreducible highest weight module $L(c_L,c_M,h_L,h_M)$ can be expressed as follows:  
		$$
		{\rm char}\, L(c_L,c_M,h_L,h_M)= q^{h_{p,r}}(1-q^{\frac{p}2})\left(1+\frac12(1+(-1)^p)q^{\frac p2}\right)(1-q^{rp})\prod_{k=1}^{\infty}\frac{1+q^{k-\frac{1}{2}}}{(1-q^{k})^{2}}.
		$$}

	\vskip 0.2cm

	Following Main Theorems 1 and 2, we derive the composition series of the Verma modules as follows:
	
	\vskip 0.2cm
	
	\noindent {\bf Main Theorem 3.}  (Theorems \ref{main4-1}, \ref{main4-2} below)
	{\it Let  $(c_L,c_M,h_L,h_M)\in\bC^4$ such that  $2h_M+\frac{p^2-1}{12}c_M=0$ for some $p\in\mathbb Z_+$ with $c_M\ne0$.

		$(1)$ If $V(c_L,c_M,h_L,h_M)$ is typical, then the Verma module $V(c_L,c_M,h_L,h_M)$ has the following infinite composition  series of submodules:
		\begin{eqnarray*}
			V(c_L,c_M,h_L,h_M)\supset\langle {\rm S}\1 \rangle \supset \langle {\rm S}^2\1 \rangle\supset\cdots\supset \langle {\rm S}^n\1 \rangle\supset \cdots, \text{ if $p$ is even};
	\\
			V(c_L,c_M,h_L,h_M)\supset \langle {\rm R}\1 \rangle\supset\langle {\rm R}^2\1 \rangle\supset\cdots\supset \langle {\rm R}^{n}\1 \rangle\supset \cdots, \text{ if $p$ is odd}.
		\end{eqnarray*} 
		
		$(2)$ If $V(c_L,c_M,h_L,h_M)$ is  atypical, 	then	the Verma module $V(c_L,c_M,h_L,h_M)$ has the following infie nit composition  series of submodules:
		$$\aligned\label{filtration-aS1}
		V(c_L,&c_M,h_L,h_M)=\langle {\rm S}^0\1 \rangle\supset\langle {\rm T}_{p, r}\1 \rangle \supset\langle {\rm S}\1 \rangle \supset \langle {\rm T}_{p, r-2}({\rm S}\1) \rangle\supset\cdots\nonumber\\ 
		&\supset\langle {\rm S}^{[\frac{r-1}2]}\1 \rangle \supset\langle {\rm T}_{p, r-2[\frac{r-1}2]}({\rm S}^{[\frac{r-1}2]}\1) \rangle\supset\langle {\rm S}^{[\frac{r-1}2]+1}\1 \rangle\supset\langle {\rm S}^{[\frac{r-1}2]+2}\1 \rangle\supset \cdots, \text{ if $p$ is even};\\
		V(c_L,&c_M,h_L,h_M)=\langle {\rm R}^0\1 \rangle\supset\langle {\rm T}_{p, r}\1 \rangle \supset\langle {\rm R}\1 \rangle \supset \langle {\rm T}_{p, r-1}{\rm R}\1 \rangle \supset  \langle {\rm R}^2\1 \rangle \supset \langle {\rm T}_{p, r-2}{\rm R}^2\1 \rangle\supset\cdots\\ 
		&\supset\langle {\rm R}^{r-1}\1 \rangle \supset\langle {\rm T}_{p, 1}{\rm R}^{r-1}\1 \rangle\supset\langle {\rm R}^{r}\1 \rangle\supset\langle {\rm R}^{r+1}\1 \rangle\supset \cdots, \text{ if $p$ is odd}.
		\endaligned  $$
	}
	\vskip 0.2cm

	As a byproduct, we also explicitly determine all singular vectors, subsingular vectors, and the composition series of Verma modules over the algebra \( W(2,2) \) (see Corollary \ref{main1-w22}, Corollary \ref{main2-w22}, and Corollary \ref{main3-w22} below).

	It is worth noting that subsingular vectors have been observed in studies of Verma modules for the N=1 Ramond algebra \cite{IK0}, the N=2 superconformal algebra \cite{DG}, and the N=1 Heisenberg-Virasoro superalgebra at level zero \cite{AJR}.
	Our main theorems reveal significant differences between the structure of Verma modules for the N=1 BMS superalgebra and those for the Virasoro algebra and N=1 Neveu-Schwarz algebra. For Verma modules over the Virasoro algebra, the maximal submodule is usually generated by two distinct weight vectors \cite{As,AF}. In contrast, the maximal submodule of a Verma module \( V(c_L, c_M, h_{p,r}, h_M) \) can always be generated by a single weight vector. Additionally, some submodules of \( V(c_L, c_M, h_L, h_M) \)  cannot be generated by singular vectors.

	The methods used to prove the main theorems differ from those in \cite{Bi, R} and also from the one in \cite{AJR, AR1}.
	Motivated by \cite{JZ}, in the present paper we introduce key operators \( {\rm S} \), \( {\rm R} \), and \( {\rm T} \), derive their crucial properties, and reveal significant relationships among them. 
	Our method of classifying  singular and subsingular vectors differs from those used for the Virasoro, super-Virasoro, and $W(2,2)$ algebras \cite{As, JZ, R}. 
	A significant advancement is the application of the total ordering on PBW bases, as defined in \cite{LPXZ}, to analyze the coefficient of the highest order terms of the  vectors $ {\rm S}\1, {\rm R}\1$ or ${\rm T}\1$ with respect to  $L_{-p}$. This approach facilitates the recursive identification of all singular and subsingular vectors. Our future research will focus on the Fock modules of the N=1 BMS superalgebra as introduced in \cite{LPXZ}, with the aim of deepening our understanding of Verma modules.

	The paper is organized as follows. In Section 2, we briefly review the relevant results on representations of the N=1 BMS superalgebra. In Section 3, in order to determine the maximal submodule of the Verma module $V(c_L,c_M,h_L,h_M)$ over $\frak g$ when it is reducible, we   investigate all singular vectors in $V(c_L,c_M,h_L,h_M)_n$ for both $n\in\mathbb Z_+$ and  $n\in\frac{1}{2}+\mathbb N$ cases. All singular vectors in $V(c_L,c_M,h_L,h_M)$ are actually determined by one element ${\rm S}$ (or ${\rm R}$) in $U(\frak{g}_-)$, see Theorems \ref{main1}, \ref{main2}.  In Section 4,  We study the quotient module of the Verma module by the submodule generated by all singular vectors determined in Section 3. In particular, we find the necessary and sufficient conditions for the existence of a subsingular vector in $V(c_L,c_M,h_L,h_M)$, and determine	
	all subsingular  vectors. See Theorems \ref{necessity}, \ref{subsingular}.
	In Section 5, we give the maximal submodules of $V(c_L,c_M,h_L,h_M)$ (which is always generated by  one weight vector) and the character formula for irreducible highest weight modules in both typical and atypical cases, see Theorems \ref{irreducibility}, \ref{irreducibility1}.
	We obtain the composition series (of infinite length)  of Verma modules $V(c_L,c_M,h_L,h_M)$ in both typical and atypical cases, see Theorems \ref{main4-1}, \ref{main4-2}.

	Throughout this  paper,  $\mathbb C$, $\mathbb N$, $\mathbb Z_+$ and $\mathbb Z$ refer to the set of complex numbers, non-negative integers, positive integers, and integers, respectively.  All vector
	spaces and algebras  are over $\mathbb C$. For a Lie (super)algebra $L$, the universal enveloping algebra of  $L$ will be denoted by
	$U(L)$. We consider a $\mathbb Z_2$-graded vector space $V = V_{\bar 0} \oplus V_{\bar 1}$, where an element $u\in V_{\bar 0}$ (respectively, $u\in V_{\bar 1}$) is called even (respectively, odd). We define $|u|=0$ if $u$ is even and $|u|=1$ if $u$ is odd. The elements in $V_{\bar 0}$ or $V_{\bar 1}$ are referred to as homogeneous, and whenever $|u|$ is used, it means that $u$ is homogeneous.

	\section{Preliminaries}
	In this section, we recall some notations and results related to the N=1 BMS superalgebra.

	\subsection{The N=1 BMS superalgebra}

	\begin{defi}\cite{BDMT}\label{Def2.1}
		The   {\bf N=1 BMS superalgebra}
		$$\frak g=\bigoplus_{n\in\mathbb{Z}}\mathbb{C} L_n\oplus\bigoplus_{n\in\mathbb{Z}}\mathbb{C} M_n\oplus\bigoplus_{n\in\mathbb{Z}+\frac{1}{2}}\mathbb{C} Q_n\oplus\mathbb{C} {\bf c}_L\oplus\mathbb{C} {\bf c}_M$$
		is a Lie superalgebra, where
		$$
		\frak g_{\bar0}=\bigoplus_{n\in\mathbb{Z}}\mathbb{C} L_n\oplus\bigoplus_{n\in\mathbb{Z}}\mathbb{C} M_n\oplus\mathbb{C}  {\bf c}_L\oplus\mathbb{C}  {\bf c}_M,\quad \frak g_{\bar1}=\bigoplus_{r\in\mathbb{Z}+\frac12} \mathbb{C}  Q_r,
		$$
		with the following commutation relations:
		\begin{align*}\label{SBMS}
			& {[L_m, L_n]}=(m-n)L_{m+n}+{1\over12}\delta_{m+n, 0}(m^3-m){\bf c}_L,\nonumber    \\
			& {[L_m, M_n]}=(m-n)M_{m+n}+{1\over12}\delta_{m+n, 0}(m^3-m){\bf c}_M,\nonumber    \\
			& {[Q_r, Q_s]}=2M_{r+s}+{1\over3}\delta_{r+s, 0}\left(r^2-\frac14\right){\bf c}_M, 
			\end{align*} 
			\begin{align*}
			& {[L_m, Q_r]}=\left(\frac{m}{2}-r\right)Q_{m+r},\nonumber                         \\
			& {[M_m,M_n]}=[M_n,Q_r]=0,                                                         \\
			& [{\bf c}_L,\frak g]=[{\bf c}_M, \frak g]=0, \nonumber
		\end{align*} for any $m, n\in\mathbb{Z}, r, s\in\mathbb{Z}+\frac12$.
	\end{defi}
	
	Note that the even part $\mathfrak{g}_{\bar 0}$ corresponds to the  BMS algebra $W(2,2)$. Additionally, the subalgebra ${{\mathfrak{vir}}} = \bigoplus_{n \in \mathbb{Z}} \mathbb{C} L_n \oplus \mathbb{C} \mathbf{c}_L$ represents the Virasoro algebra.

	The N=1 BMS superalgebra $\mathfrak{g}$ has a $\frac{1}{2}\mathbb{Z}$-grading by the eigenvalues of the adjoint action of $L_0$. It is clear that   $\mathfrak{g}$ has the following triangular decomposition:
	\begin{eqnarray*}
		\mathfrak{g}={\mathfrak{g}}_{-}\oplus {\mathfrak{g}}_{0}\oplus {\mathfrak{g}}_{+},
	\end{eqnarray*}
	where
	\begin{eqnarray*}
		&&{\mathfrak{g}}_{\pm}=\bigoplus_{n\in \mathbb{Z}_+}\bC L_{\pm n}\oplus \bigoplus_{n\in \mathbb{Z}_+}\bC M_{\pm n}\oplus \bigoplus_{r\in \frac{1}{2}+\mathbb{N}}\bC Q_{\pm r},\\
		&&{\mathfrak{g}}_{0}=\bC L_0\oplus\bC M_0\oplus\bC {\bf c}_L\oplus
		\bC {\bf c}_M.
	\end{eqnarray*}

	\subsection{Verma modules}
	For $(c_L,c_M,h_L,h_M)\in\bC^4$, let $\bC$ be the module over 
	${\mathfrak{g}}_{0}\oplus{\mathfrak{g}}_{+}$  defined  by
	\begin{eqnarray*}
		{\bf c}_L{ 1}=c_L{ 1},\quad {\bf c}_M{  1}=c_M{ 1},\quad L_0{ 1}=h_L{ 1},\quad  M_0{ 1}=h_M{ 1},\quad{\mathfrak{g}}_{+}1=0.
	\end{eqnarray*}   
	The Verma module over ${\mathfrak{g}}$ is defined as follows 
	\begin{eqnarray*}
		V(c_L,c_M,h_L,h_M)=U({\mathfrak{g}})\ot_{U({\mathfrak{g}}_{0}\oplus{\mathfrak{g}}_{+})}\bC\simeq U({\mathfrak{g}}_{-})\1,
	\end{eqnarray*}
	where $\1=1\ot 1$.
	It follows that
	$V(c_L,c_M,h_L,h_M)=\bigoplus_{n\in \mathbb{N}}V(c_L,c_M,h_L, h_M)_{n}$ and $U({\mathfrak{g}_-})=\bigoplus_{n\in \mathbb{Z}_+}U(\mathfrak{g}_-)_{n},$
	where
	$$V(c_L,c_M,h_L,h_M)_{n}
	=\{v \in V(c_L,c_M,h_L,h_M)\,|\,L_0v =(h_L+n)v\}
	$$ and
	$$U(\mathfrak{g}_-)_{n}
	=\{x \in U(\mathfrak{g}_-)\,|\,[L_0, x]=nx\}.
	$$
	Moreover,
	$
	V(c_L,c_M,h_L,h_M)=V(c_L,c_M,h_L,h_M)_{\bar{0}}\oplus V(c_L,c_M,h_L,h_M)_{\bar{1}}$
	with
	$$\aligned V(c_L,c_M,h_L,h_M)_{\bar{0}}=&\bigoplus_{n\in\mathbb{N}}V(c_L,c_M,h_L,h_M)_{n},\\ V(c_L,c_M,h_L,h_M)_{\bar{1}}=&\bigoplus_{n\in \mathbb{N}}V(c_L,c_M,h_L,h_M)_{\frac{1}{2}+n}.\endaligned$$

	It is clear that $V(c_L,c_M,h_L,h_M)$ has
	a unique maximal submodule $J(c_L,c_M,h_L,h_M)$ and the factor module
	$$
	L(c_L,c_M,h_L,h_M)=V(c_L,c_M,h_L,h_M)/J(c_L,c_M,h_L,h_M)
	$$
	is an irreducible highest weight   ${\mathfrak{g}}$-module. Define
	$${\rm char}\, V(c_L,c_M,h_L,h_M)=q^{h_L}\sum_{i\in\frac12\mathbb N }{\rm dim}\, V(c_L,c_M,h_L,h_M)_iq^{i}.$$
	
	An eigenvector   $u$ in $V(c_L,c_M,h_L,h_M)$ with respect to $\mathfrak{g}_0$ is called a {\bf singular vector} if $\mathfrak{g}_{+} u=0$.
	Let $J'(c_L,c_M,h_L,h_M)$ be the submodule of $V(c_L,c_M,h_L,h_M)$ generated by all singular vectors.
	A weight vector $u'$ in $V(c_L,c_M,h_L,h_M)$ is called a {\bf subsingular vector} if $u'+J'(c_L,c_M,h_L,h_M)$ is a singular vector in $V(c_L,c_M,h_L,h_M)/J'(c_L,c_M,h_L,h_M)$.

	Recall that a partition of a positive integer $n$ is a finite non-increasing sequence of
	positive integers $\la=(\la_1,\la_2,\dots, \la_r)$ such that $n=\sum_{i=1}^r\la_i$. The positive integer $\la_i$ is called the $i$-th entry of the partition $\la$. We call $r$ the length of $\la$, denoted by $\ell(\la)$, and call the sum of $\la_i$'s the weight of $\la$, denoted
	by $|\la|$. Denote $\la-\frac12=\left(\la_1-\frac12,\la_2-\frac12,\dots, \la_r-\frac12\right)$ and  $-\la=(-\la_1,-\la_2,\dots, -\la_r)$.
	The number of partitions of $n$ is given by the partition function ${\tt p}(n)$.  Denote by $\mathcal P$ the set of all partitions (including the empty partition) and
	$\mathcal P(n)$ the set of all partitions with   weight   $n\in\mathbb Z_+$.
	A partition $\la=(\la_1,\la_2,\dots, \la_r)$ is called strict if $\la_1 >\la_2 >\dots >\la_r >0$. The set $\mathcal{SP}$ consists of all strict partitions (including the empty partition). Recall that the natural ordering on $\mathcal P$ and $\mathcal{SP}$ is defined as follows:
	\begin{eqnarray*}
		&&\la> \mu\iff |\la|> |\mu|, \text{ or } |\la|= |\mu|, \la_1=\mu_1,\dots, \la_k=\mu_k, \text{ and }\la_{k+1}>\mu_{k+1} \text{ for some }\ k\geq0;\\
		&&\la=\mu\iff \la_i=\mu_i \quad\text{for all }\ i.
	\end{eqnarray*}

	According to the Poincar\'{e}-Birkhoff-Witt ($\mathrm{PBW}$)  theorem, every vector $v$ of $V(c_L,c_M,h_L,h_M)$ can be uniquely written in
	the following form
	\begin{equation}\label{def2.1}
		v=\sum_{\lambda, \nu\in\mathcal P, \mu\in\mathcal{SP}}a_{\lambda, \mu, \nu}M_{-\la}Q_{-\mu+\frac12}L_{-\nu}{\bf 1},
	\end{equation}
	where $a_{\lambda, \mu, \nu}\in\mathbb C$ and only finitely many of them are non-zero, and
	$$M_{-\la}:=M_{-\la_1}\cdots M_{-\la_r},\  Q_{-\mu+\frac12}:=Q_{-\mu_1+\frac12}\cdots Q_{-\mu_s+\frac12},\  L_{-\nu}:=L_{-\nu_1}\cdots L_{-\nu_t}.$$ For any $v\in V(c_L,c_M,h_L,h_M)$ as in  \eqref{def2.1}, we denote by $\mathrm{supp}(v)$ the set of all $(\lambda, \mu, \nu)\in \mathcal P\times\mathcal{SP}\times \mathcal P$  such that $a_{\lambda, \mu, \nu}\neq0$.
	Next, we define
	\begin{eqnarray*}
		&&\mathcal{M}={\rm span}_{\mathbb C}\{M_i \mid i\in\mathbb Z\},\quad \mathcal{M}_-={\rm span}_{\mathbb C}\{M_{-i} \mid i\in\mathbb Z_+\},\\
		&&\mathcal{Q}={\rm span}_{\mathbb C}\{Q_{-i+\frac12}\mid i\in \mathbb Z\},\ \mathcal{Q}_-={\rm span}_{\mathbb C}\{Q_{-i+\frac12}\mid i\in \mathbb Z_+\}.
	\end{eqnarray*}
	Note that $\mathcal{M}+\mathbb{C} {\bf c}_M$ and $ \mathcal{M}+\mathcal{Q}+\mathbb{C} {\bf c}_M$ are ideals of $\mathfrak{g}$.

	For $y=M_{-\la}Q_{-\mu+\frac12}L_{-\nu}$ or $y\1$, we define  $$\ell(y)=\ell(y\1)=\ell(\lambda)+\ell(\mu)+\ell(\nu), \  {\rm deg}(y)=|\la|+|\mu-\frac12|+|\nu|.$$
	
	For \eqref{def2.1}, we define
	$${\ell}_M(v):={\rm max}\{\ell(\lambda)\mid (\lambda, \mu, \nu)\in {\rm supp}(v)\}.$$
	Similarly, we can define ${\ell}_Q(v)$, ${\ell}_L(v)$ and ${\rm deg}(v)$.

	For $n\in \frac{1}{2}\mathbb Z_+$, let
	$$
	B_{n}=\{M_{-\la}Q_{-\mu+\frac12}L_{-\nu}{\bf 1}\mid |\la|+|\mu-\frac12|+|\nu|=n, \forall\  \la,\nu\in \mathcal P, \mu\in\mathcal{SP} \}.
	$$
	Clearly, $B_{n}$ is a basis of $V(c_L,c_M,h_L,h_M)_n$. Then
	$$
	|B_{n}|=\dim V(c_L,c_M,h_L,h_M)_{n},
	$$
	\begin{eqnarray*}\label{2.6}
		{\rm char}\, V(c_L,c_M,h_L,h_M)=q^{h_L}\prod_{k=1}^{\infty}\frac{1+q^{k-\frac{1}{2}}}{(1-q^{k})^{2}}.
	\end{eqnarray*}

	Now we can define a total ordering $\succ$ on  $B_{n}$:
	$M_{-\la}Q_{-\mu+\frac{1}{2}}L_{-\nu}{\bf 1} \succ
	M_{-\la'}Q_{-\mu'+\frac{1}{2}}L_{-\nu'}{\bf 1}$ if and only if one of the following condition is satisfied:
	\begin{itemize}
		\item[(i)]$|\nu|>|\nu'|;$
		\item[(ii)]$|\nu|=|\nu'|\ \mbox{and}\
		\ell(\nu)>\ell(\nu')$;
		\item[(iii)]$|\nu|=|\nu'|,
		\ell(\nu)=\ell(\nu')\ \mbox{and}\ \nu>\nu'$;
		\item[(iv)]$\nu=\nu',\  
		\mu>\mu';$
		\item[(v)]$\nu=\nu',\ \mu=\mu',\ \mbox{and}\
		\la>\la'.$
	\end{itemize}

	Let
	\begin{eqnarray*}
		B_{n}=\{b_i\mid b_{i}\succ b_{j}\  \text{for}\  i>j\},\quad\text{where}\quad b_{i}=M_{-\la^{(i)}}G_{-\mu^{(i)}}L_{-\nu^{(i)}}\1,
	\end{eqnarray*}
	with $\la^{(i)},\nu^{(i)}\in \mathcal P ,\mu^{(i)}\in \mathcal{SP}$ and $|\la^{(i)}|+|\mu^{(i)}-\frac{1}{2}|+|\nu^{(i)}|=n$ for any $i$.
	

	Any non-zero homogenous vector $X\in V_n=V(c_L,c_M,h_L,h_M)_{n}$  can be uniquely written  as a linear
	combination of elements in $B_{n}$ for some $n\in\mathbb Z_+$: 
	$$X=\Sigma_{i=1}^m a_iX_i,\text{ where }
	0\neq a_i\in\mathbb C, X_i\in B_{n}\text{ and }X_1\succ X_2\succ\cdots\succ
	X_m.$$ We define the {\bf highest term} of $X$ as ${\rm hm}(X)=X_1$.

	Now we define on $V(c_L,c_M,h_L,h_M)$ the operations of formal partial derivative $\frac{\partial}{\partial Q_{- i+\frac12}}, i\in \mathbb{Z}_+$ as follows
	\begin{eqnarray*}		
		\frac{\partial}{\partial Q_{- i+\frac12}}M_{- j}=\frac{\partial}{\partial Q_{- i+\frac12}}L_{- j}=0,\ \frac{\partial}{\partial Q_{- i+\frac12}}Q_{- j+\frac12}=\delta_{ji},\ \frac{\partial}{\partial Q_{- i+\frac12}}\1=0\end{eqnarray*}
	and then define their actions on monomials (\ref{def2.1}) by the super-Leibniz rule. Finally, we
	extend these to $U(\frak{g}_-)$ by linearity.
	
	Let us recall the necessary and sufficient conditions for the Verma module $V(c_L,c_M,h_L,h_M)$ to be irreducible.
	
	\begin{theo} \label{Sim} \cite[Theorem 3.2]{LPXZ}
		For $(c_L,c_M,h_L,h_M)\in\bC^4$, the Verma module $V(c_L,c_M,h_L,h_M)$ over $\mathfrak{g}$ is irreducible if and only if
		$$
		2h_M+\frac{i^2-1}{12}c_M\neq 0,\ \forall  i\in \mathbb Z_{+}.
		$$
	\end{theo}
	
	From now on 
	we always assume that $$\phi(p)=2h_M+\frac{p^2-1}{12}c_M=0$$ for some $p\in \mathbb Z_+$.  This assumption indicates  that the Verma module $V(c_L,c_M,h_L,h_M)$ is reducible and  contains a
	singular vector not in $\mathbb{C}{\bf 1}$.

	\begin{theo}\label{cor3.3}\cite[Theorem 3.3 and Proposition 5.2]{LPXZ} Suppose $(c_L,c_M,h_L,h_M)\in\bC^4$ such that $\phi(1)=0$, which implies that $h_M=0$.   
		\begin{itemize}
			\item[$(1)$] The vectors $M_{-1}{\bf 1}$ and  $Q_{-\frac{1}{2}}{\bf 1}$ of $V(c_L,c_M,h_L,0)$ are singular vectors. If further $h_L =0$, then $L_{-1}{\bf 1}$ is a subsingular vector of  $V(c_L,c_M,0,0)$, i.e., a  singular vector of \\ $V(c_L,c_M,0,0)/\langle Q_{-\frac12}\mathbf 1\rangle$, where $\langle Q_{-\frac12}\mathbf 1\rangle$ is the $\mathfrak{g}$-submodule generated by $Q_{-\frac12}\mathbf 1$.
			
			\item[$(2)$] 
			The vaccum  module $V(c_L,c_M)=V(c_L,c_M,0,0)/\langle L_{-1}\1\rangle$ is irreducible if and only if $c_M\neq 0$.  	
			
			\item[$(3)$] The vacuum module $V(c_L,c_M)$ with $c_M\ne0$ is endowed with a simple vertex superalgebra structure. There is a one-to-one correspondence between smooth $\mathfrak{g}$-modules of central charge $(c_L, c_M)$ and   $V(c_L,c_M)$-modules.
		\end{itemize}
	\end{theo}
	
	The following result is obvious.
	
	\begin{lem}\label{degenerated-case} If   $ c_M=0$ and $h_M=0$, then the Verma module $V(c_L,0,h_L,0)$ possesses a submodule   $ \langle \mathcal{M}_-\mathbf1, \mathcal{Q}_-\mathbf1\rangle$ and the quotient module $V(c_L,0,h_L,0)/  \langle \mathcal{M}_-\mathbf1, \mathcal{Q}_-\mathbf1\rangle$ is isomorphic to the Verma module $V_{\mathfrak{vir}}(c_L,h_L)$ over the Virasoro algebra.
	\end{lem}
	
	For the remaining case $p=1$ and  $c_M\ne0$ (in this case $h_M=0$),
	the structure of the Verma module $V(c_L,c_M,h_L,0)$ will be determined in the next sections.


	\section{Classification of singular vectors of Verma modules}
	
	Fix  $(c_L,c_M,h_L,h_M)\in\bC^4$. In this section we will determine all singular vectors of the Verma module $V(c_L,c_M,h_L,h_M)$ when it is reducible.
	
	From now on we will   assume that $\phi(p)=0$ for some $p\in\mathbb Z_+$ with $c_M \ne0$. The case $c_M=0$ (then $h_M=0$) was solved in Theorem \ref{cor3.3} and Lemma \ref{degenerated-case}.

	\subsection{Singular vectors in $V(c_L,c_M,h_L,h_M)_n$ for $ n\in\mathbb Z_+$}


	First, we construct a singular vector ${\rm S}\1$ in $V(c_L,c_M,h_L,h_M)_n$ for some $ n\in\mathbb Z_+$.
	
	\begin{pro}
		\label{singular-S1}  The Verma module
		$V(c_L,c_M,h_L,h_M)$ possesses a singular vector 
		\begin{eqnarray}\label{e3.7}
			u={\rm S}\1=M_{-p}\1+\sum_{\mu\in \mathcal P(p), \lambda<(p) }s_{\mu}M_{-\mu}\in U(\mathfrak{g}_{-})\1\in
			V(c_L,c_M,h_L,h_M)_p, 
		\end{eqnarray} where
		$$
		s_{\mu}=(-1)^{\ell(\mu)-1}\prod_{i=1}^{\ell(\mu)-1}\frac{2(p-\sum_{j=0}^{i-1}\mu_j)-\mu_{i}}{2(p-\sum_{j=1}^{i}\mu_j)\phi(p-\sum_{j=1}^{i}\mu_j)},
		$$
		and $\mu_0=0$, $\mu=(\mu_1, \mu_2, \cdots, \mu_s)\in\mathcal P(p)$.
		
	\end{pro}
	\begin{proof} Suppose that $${\rm S}=\sum_{\lambda\in \mathcal P(p) }s_{\lambda}M_{-\lambda}\in U(\mathfrak{g}_{-}), s_{\lambda}\in \mathbb{C},$$  
		where the ordering of all summands of ${\rm S}$ is according to "$\succ$" defined in Section 2.2 as follows
		\begin{eqnarray*}
			M_{-p}, M_{-(p-1)}M_{-1}, M_{-(p-2)}M_{-2}, M_{-(p-2)}M_{-1}^{2}, \cdots, M_{-1}^p.\label{W-ordering}
		\end{eqnarray*}
		
		Now we consider the ${\tt p}(p)$ linear equations:	
		\begin{eqnarray}
			L_{p}u=0,\  L_{p-1}L_{1}u=0,\  L_{p-2}L_{2}u=0,\  L_{p-2}L_{1}^{2}u=0,\ \cdots, L_{1}^{p}u=0.\label{o2.110}
		\end{eqnarray}
		%
		%
		%
		%
		%
		%
		
		The coefficient matrix $A_{p}$ of the linear equations \eqref{o2.110} is
		$$A_{p}=\left(
		\begin{array}{ccccc}
			p\phi(p) & 0 & 0 & \cdots & 0 \\
			\star & \star & 0 & \cdots & 0 \\
			\star & \star & \star & \cdots & 0 \\
			\vdots &  \vdots &  \vdots & \cdots &  \vdots \\
			\star & \star & \star & \cdots & \star \\
		\end{array}
		\right).
		$$

		Clearly, $A_{p}$ is an lower triangular whose first row is zero, its other diagonal entries and other entries in the first column $\star$ are non-zero. So there exists a unique solution for ${\rm S}$ with $1$ as the coefficient of $M_{-p}$ up to a scalar multiple.
		
		Certainly, by the actions of $L_i, i=p-1, p-2, \cdots, 1$ on $u={\rm S}\1$ we can get all $s_{\mu}$.
	\end{proof}

	\begin{exa}
		\begin{eqnarray*}
			&(1)&p=1,h_M=0: {\rm S}=M_{-1};\\
			&(2)&p=2,h_M=-\frac{1}{8}c_M: {\rm S}=M_{-2}+\frac{6}{c_M}M_{-1}^2;\\
			&(3)&p=3,h_M=-\frac{1}{3}c_M: {\rm S}=M_{-3}+\frac{6}{c_M}M_{-2}M_{-1}+\frac{9}{c_M^2}M_{-1}^3.
		\end{eqnarray*}
	\end{exa}

	\begin{lem}\label{l3.15'} For any $x\in \frak g_+$, we have
		\begin{eqnarray*}
			[x, {\rm S}]\subset U(\mathcal{M}_-) \left(2M_0+\frac{p^2-1}{12}{\bf c}_M\right)+U(\mathcal{M}_-)\mathcal M_+.
		\end{eqnarray*}
	\end{lem}
	\begin{proof} 
		It follows from the fact that $[x, {\rm S}]\1=0$ in $V(c_L,c_M,h_L,h_M)$  and $[x, {\rm S}]\in U(\mathcal{M}\oplus \mathbb C{\bf c}_M) $ for any $x\in\frak g_+$.
	\end{proof}

	\begin{lem}\label{singular-Sk}
		Let $u={\rm S}{\bf 1}$ be the singular vector in  Proposition \ref{singular-S1}.
		Then ${\rm S}^k{\bf 1}$ is   a singular vector in $V(c_L,c_M,h_L,h_M)_{kp}$ for any $k\in\mathbb Z_+$.
	\end{lem}
	\proof  It follows from Lemma \ref{l3.15'}.\qed
	%
	%

	\begin{lem}\label{l3.6}
		If $u$ is a singular vector in $V(c_L,c_M,h_L,h_M)_n$ for $n\in \mathbb Z_+$ with $\ell_L(u)=0$, then $\ell_Q(u)=0$.
	\end{lem}
	\begin{proof} Assume that $\ell_Q(u)\ne0$.
		Set
		\begin{eqnarray*}
			u=\sum_{\mu\in\mathcal {SP}}a_\mu Q_{-\mu+\frac12}\1\in V(c_L,c_M,h_L,h_M)_n,
		\end{eqnarray*}
		where $a_\mu \in U(\mathcal M_-)$.
		Take $\bar\mu=(\bar\mu_1, \cdots, \bar\mu_s)$ among all $\mu$ with $a_{\mu}\ne0$ such that ${\bar\mu}_1$ is maximal.
		
		Certainly, $s\ge 2$ since $\ell(\bar\mu)$ is even and  $n\in\mathbb Z$.
		
		\begin{eqnarray*}
			0=Q_{{\bar\mu}_1-\frac12}u=\left(2h_M+\frac{(2\bar \mu_1-1)^2-1}{12}c_M\right)\sum_{\mu_1=\bar{\mu_1}} a_{\mu}\frac{\partial}{\partial Q_{-\bar\mu_1+\frac12}}Q_{-\mu+\frac12}{\bf 1}.
		\end{eqnarray*}
		If $2\bar\mu_1-1\ne p$, then $Q_{{\bar\mu}_1-\frac12}u\ne 0$, which is  a contradiction. 
		
		Now we only consider the case of $p=2\bar\mu_1-1$ being odd, and in this case $p>2\bar\mu_2-1$.
		
		By acting $Q_{\bar\mu_2-\frac12}$ on $u$, we get
		\begin{eqnarray*}
			0=Q_{{\bar\mu}_2-\frac12}u=\left(2h_M+\frac{(2\bar \mu_2-1)^2-1}{12}c_M\right)\sum_{\mu_1=\bar{\mu_1}} a_{\mu}\frac{\partial}{\partial Q_{-\bar\mu_2+\frac12}}Q_{-\mu+\frac12}{\bf 1}+B\ne 0,
		\end{eqnarray*}
		where $\nu_1<\bar{\mu}_1$ for all summand $a_\nu Q_{-\nu+\frac12}$ in $B$ with $a_\nu\ne0$.
		It also gets a contradiction.
	\end{proof}
	
	Now we shall determine all singular vectors in $V(c_L,c_M,h_L,h_M)$ with ${\ell}_L(u)=0$.
	
	\begin{theo}\label{singular-W} Let  $(c_L,c_M,h_L,h_M)\in\bC^4$ such that  $\phi(p)=0$  with $c_M\ne 0$,  ${\rm S}$ be defined in  Proposition \ref{singular-S1} and $u\in V(c_L,c_M,h_L,h_M)_n$  for some $p\in \mathbb Z_+$ with ${\ell}_L(u)=0$.  Then $u$ is   a 
		singular vector  if and only if $n=kp$ for some $k\in\mathbb Z_+$. In this case  $u={\rm S}^k{\bf 1}$ up to a scalar multiple if $n=pk$.
	\end{theo}
	\proof

	Let $u\in V(c_L,c_M,h_L,h_M)_n$ be a singular vector. By Lemma \ref{l3.6} we can suppose that
	\begin{equation}
		u= (a_0{\rm S}^k+a_1{\rm S}^{k-1}+\cdots a_{k-1}{\rm S}+a_k)\mathbf1,\label{E3.5}
	\end{equation}
	where $k\in\mathbb Z_+$ and  each $a_i\in U(\mathcal M_-)$ does not involve $M_{-p}$ for any $i=0,1, \cdots, k$. We may assume that $a_0\ne0$.
	
	If ${\rm hm}\{a_0, a_1, \cdots, a_k\}\notin \mathbb C$, set ${\rm hm}\{a_0, a_1, \cdots, a_k\}=M_{-\lambda}$. By the action of $L_{\lambda}$ on (\ref{E3.5}),  we can get 
	$L_{\lambda}u\ne0$ since all $a_i\in U(\mathcal M_-)$ are not involving $M_{-p}$.
	So $a_0\in\mathbb C$ and $u=a_0{\rm S}^k{\bf 1}.$  The theorem follows.
	\qed

	\begin{lem}\label{l3.1}  If $u$ is a singular vector in $V(c_L,c_M,h_L,h_M)_n$, then $\ell_{L}(u)=0$.
	\end{lem}
	\begin{proof}
		To the contrary we assume that  $\ell_{L}(u)\neq 0$. Write $u=y\1\in V(c_L,c_M,h_L,h_M)_{n}$, where $y\in U(\mathfrak{g}_-)$. Then $M_0 u=M_0y\1=[M_0,y]\1+h_Mu$. If $[M_0,y]\neq 0$, then $\ell_{L}([M_0,y])<\ell_{L}(y)$. This implies $[M_0,y]\neq ay $ for any $a\in\mathbb{C}^*$, showing that $u$ is not a singular vector of $V(c_L,c_M,h_L,h_M)$.
		So $[M_0,y]=0$.
		
		%

		We write
		\begin{equation*}
			u=y\1= (a_0L_{-p}^k+a_1L_{-p}^{k-1}+a_2L_{-p}^{k-2}+\cdots+a_k)\1, \label{singularL}
		\end{equation*}
		where $k\in\mathbb Z_+, a_i\in U(\frak g_-), i=0, 1, \cdots, k, a_0\ne 0$ and any $a_i$ does not involve $L_{-p}$.
		
		We claim that ${\ell}_L(a_0)=0$. 
		Otherwise, ${\rm hm}(a_0)=a'_0L_{-\nu}$ for some $a'_0\in \mathcal{MQ}$ and $\emptyset\ne\nu\in\mathcal P$.
		Then $[M_{\nu}, y]\1=a'_0[M_{\nu}, L_{-\nu}]L_{-p}^k\1+a'_1L_{-p}^{k-1}\1+\cdots+a'_k\1\ne 0$, where $a'_i\in U(\frak g_-),  i=1, \cdots, k$ with any $a'_i$ not involving $L_{-p}$. This is a contradiction since $[M_{\nu}, L_{-\nu}]\ne 0$ by the assumption of $L_{-\nu}$ not involving $L_{-p}$.
		This claim follows and $a_0\in U(\mathcal M_-+\mathcal Q_-)$.

		Now we shall prove that $k=0$ and get the lemma.
		To the contrary, assume that $k\ge 1$.
		In the case if $[M_0, a_1]=0$ we see that \begin{equation*} [M_0, y]= kpa_0M_{-p}L_{-p}^{k-1}+A=0, \label{M0-act}\end{equation*}
		where the degree of $L_{-p}$ in $A$ is no more than $k-2$.  
		It is a contradiction.
		
		So $[M_0, a_1]\ne0$. If ${\ell}_L(a_1)\ge 2$, then 
		\[ [M_0, y]= kpa_0M_{-p}L_{-p}^{k-1}+[M_0, a_1]L_{-p}^{k-1}+B=0, \] where the degree of $L_{-p}$ in $B$ is no more than $k-2$.
		We see that  ${\ell}_L([M_0, a_0L_{-p}^{k}])=k-1$, ${\ell}_L([M_0, a_1]L_{-p}^{k-1})\ge k$, yielding that $[M_0, y]\ne0$, which is a contradiction.
		
		Now we obtain that  ${\ell}_L(a_1)=1$ and set
		\begin{equation*}a_1=\sum_{i=1}^{s}b_iL_{-i}+b_0, \label{eqa1}\end{equation*} where $b_i\in U(\mathcal{M}_-+\mathcal Q_-)$ and $b_s\ne 0$.
		$$[M_s, y]=a_0[M_s, L_{-p}^k]+b_1[M_s, L_{-s}]L_{-p}^{k-1}+B',$$ where the degree of $L_{-p}$ in  $B'$ is no more than $k-2$.
		
		If $s>p$,  then  ${\ell}_L(a_0[M_s, L_{-p}^k])\le k-2$ and ${\ell}_L([M_s, L_{-s}]L_{-p}^{k-1})=k-1$.  In this case $[M_s, y]\1\ne 0$, it is a contradiction.
		So $s<p$. 
		
		Note that if $p=1$, then $s=0$, which means ${\ell}_L (a_1)=0$. This is a contradiction.

		So we can suppose that $p>1$. By action of $L_i$ for any $i\in\mathbb Z_+$ on $u$ we get
		$$L_iu= L_{-p}^k[L_i, a_0]\1+A=0, $$
		where the degree of $L_{-p}$ in $A$ is no more than $k-1$. So $[L_i, a_0]\1=0$ for any $i\in\mathbb Z_+$.
		In this case, $a_0\1$ becomes a singular vector of $V(c_L,c_M,h_L,h_M)$ with ${\ell}_L(a_0\1)=0$.

		By Theorem \ref{singular-W}, we get
		$	a_0=d_0{\rm S}^l $
		where $l\in\mathbb N,  d_0 \in\mathbb C^*$.
		In this case,
		\begin{equation*}[M_0, y]\1=kpa_0M_{-p}L_{-p}^{k-1}\1+[M_0, a_1]L_{-p}^{k-1}\1+B\1=0,\label{eqMp}\end{equation*} where the degree of $L_{-p}$ in $B$ is no more than $k-2$.
		So
		\begin{equation}kpd_0{\rm S}^lM_{-p}+[M_0, a_1]=0.\label{eqMp1}\end{equation}
		By considering the degree of ${\rm S}$ in \eqref{eqMp1},  we have $a_1=f_0{\rm S}^{l+1}+f_2{\rm S}^l+\cdots+f_{l+1}$, where $f_i\in U(\frak g_-)$ not involving $L_{-p}, M_{-p}$.
		Comparing the coefficients of ${\rm S}^{l+1}$ in \eqref{eqMp1}, we get $$[M_0, f_0]=kpd_0\in\mathbb C^*,$$
		a contradiction. \end{proof}

	\begin{theo} \label{main1} Let  $(c_L,c_M,h_L,h_M)\in\bC^4$ such that  $\phi(p)=0$  with $c_M\ne 0$. Let ${\rm S}\1$ be the singular vector  in  Proposition \ref{singular-S1}, $n\in\mathbb Z_+$. Then  $V(c_L,c_M,h_L,h_M)_n$ possesses a singular vector $u$ if and only if $n=kp$ for some $k\in\mathbb Z_+$. In this case  $u={\rm S}^k{\bf 1}$ if $n=pk$ up to a scalar multiple.
	\end{theo}
	\begin{proof}
		It follows from Theorem \ref{singular-W} and Lemma \ref{l3.1}.
	\end{proof}

	\subsection{Singular vectors in $V(c_L,c_M,h_L,h_M)_{n-\frac12}$ for $ n\in\mathbb Z_+$}
	
	In this subsection, we shall determine all singular vectors in $V(c_L,c_M,h_L,h_M)_{n-\frac12}$ for $ n\in\mathbb Z_+$.
	\begin{lem}\label{singular-Q1}
		If there exists a singular vector in $V(c_L,c_M,h_L,h_M)_{n-\frac12}$ for $ n\in\mathbb Z_+$ with $\ell_{L}(u)=0$, then $p$ is odd and $\ell_Q(u)=1$.
	\end{lem}
	
	\proof 
	It follows similar arguments as in that of  Lemma \ref{l3.6} and the fact that $\ell_Q(u)\ge 1$ here.
	%
	%
	%
	%
	%
	\qed

	\begin{pro}
		\label{singular-R1} Let $p\in 2\mathbb Z_+-1$. Then
		the Verma module
		$V(c_L,c_M,h_L,h_M)$ possesses a singular vector
		$u\in
		V(c_L,c_M,h_L,h_M)_{\frac{p}{2}}$ with $\ell_{L}(u)=0$. Up to a scalar multiple, it is unique and can be written as
		\begin{eqnarray}
			u={\rm R}\1=Q_{-\frac{p}{2}}\1+\sum_{i=1}^{\frac{p-1}{2}}f_{i}(M)Q_{-\frac{p}{2}+i}\1,
		\end{eqnarray}
		where $f_i(M)=\sum_{|\lambda|=i}c_{\lambda}M_{-\lambda}$ for some $c_{\lambda}\in \mathbb{C}$.
	\end{pro}
	\proof It suffices to prove the case of $p>1$. 
	
	By Lemma \ref{singular-Q1}, we can suppose that
	\begin{eqnarray*}
		{\rm R}=f_0Q_{-\frac{p}{2}}+\sum_{i=1}^{\frac{p-1}{2}}f_i(M)Q_{-i+\frac{1}{2}},
	\end{eqnarray*}
	where $ f_0\in \mathbb C, f_i(M)\in U(\mathcal M_-), i=1, 2, \cdots, \frac{p-1}{2}$.
	
	Here the ordering  of all summands of ${\rm R}$ is according to the ordering $\succ$ defined in Section 2.2 as follows
	\begin{eqnarray*}
		Q_{-\frac{p}{2}}, M_{-1}Q_{-\frac{p}{2}+1}, M_{-2}Q_{-\frac{p}{2}+2}, M_{-1}^{2}Q_{-\frac{p}{2}+2}, \cdots, M_{-1}^{\frac{p-1}{2}}Q_{-\frac{1}{2}}.\label{o2.10}
	\end{eqnarray*}
	Now we consider the following linear equations.
	\begin{eqnarray}\label{eee4.8}
		Q_{\frac{p}{2}}u=L_{1}Q_{\frac{p}{2}-1}u=L_{2}Q_{\frac{p}{2}-2}u=L_{1}^{2}Q_{\frac{p}{2}-2}u=\cdots=L_{1}^{\frac{p-1}{2}}Q_{\frac{1}{2}}u=0.
	\end{eqnarray}
	
	The number of these equations is exactly $\sum_{i=0}^{\frac{p-1}{2}}{\tt p}(i)$. 
	By direct calculations we can see that the coefficient  matrix $A_{p}$ of \eqref{eee4.8} is lower triangular and the first row is zero.  All other diagonal elements are non-zero's by assumption. So there exists a unique solution with a non-zero coefficient of $Q_{-\frac p2}$ up to a scalar multiple.  The proposition follows.
	\qed

	In the following, we provide an explicit formula for ${\rm R}$.
	\begin{pro}\label{singular-R11}
		Let $p\in2\mathbb Z_+-1$. Then the singular vector ${\rm R}\1$ in  Proposition \ref{singular-R1} can be determined as 
		\begin{eqnarray}\label{R-exp}
			{\rm R}\1=Q_{-\frac{p}{2}}\1+\sum_{i=1}^{\frac{p-1}{2}}f_{i}(M)Q_{-\frac{p}{2}+i}\1,
		\end{eqnarray}
		where  
		\begin{eqnarray}
			f_{1}(M)=c_1M_{-1},
			f_{i}(M)=c_iM_{-i}+\sum_{j=1}^{i-1}c_if_j(M)M_{-(i-j)},
		\end{eqnarray}
		and $c_i=\frac{6}{i(p-i)c_M}$ for $i=1,\cdots,\frac{p-1}{2}$. 
	\end{pro}
	\begin{proof}
		Let ${\rm R}\1$ be as \eqref{R-exp}, a singular vector in $V(c_L,c_M,h_L,h_M)_{\frac p2}$,
		where $f_{i}(M)\in U(\mathcal{M}_-)$ with degree $i$, $i=1,\cdots,\frac{p-1}{2}$.

		For $i=1, 2, \cdots,\frac{p-1}{2}$, using the action of $Q_{\frac{p}{2}-i}$ on \eqref{R-exp}, we deduce that
		\begin{eqnarray*}
			0=Q_{\frac{p}{2}-i}{\rm R}\1&=&[Q_{\frac{p}{2}-i},Q_{-\frac{p}{2}}]\1+\sum_{j=1}^if_j(M)[Q_{\frac{p}{2}-i},Q_{-\frac{p}{2}+j}]\1 \\
			&=&2M_{-i}\1+2f_1(M)M_{-i+1}\1+\cdots+f_{i}(M)\left(2M_0+\frac{(p-2i)^2-1}{12}{\bf c}_M\right)\1.
		\end{eqnarray*}
		Applying $h_M=-\frac{p^2-1}{24}c_M$ we deduce that 
		$$2M_{-i}+2f_1(M)M_{-i+1}+\cdots-f_{i}(M) \frac{i(p-i)}{3}c_M =0,$$
		\begin{eqnarray*}
			f_{1}(M)=c_1M_{-1},
			f_{i}(M)=c_iM_{-i}+\sum_{j=1}^{i-1}c_if_j(M)M_{-(i-j)},
		\end{eqnarray*}
		and $c_i=\frac{6}{i(p-i)c_M}$ for $i=1,\cdots,\frac{p-1}{2}$. 
	\end{proof}

	\begin{exa}
		\begin{eqnarray*}
			&(1)&p=1,h_M=0: {\rm R}=Q_{-\frac{1}{2}};\\
			&(2)&p=3,h_M=-\frac{1}{3}c_M: {\rm R}=Q_{-\frac{3}{2}}+\frac{3}{c_M}M_{-1}Q_{-\frac{1}{2}};\\
			&(3)&p=5,h_M=-c_M: {\rm R}=Q_{-\frac{5}{2}}+\frac{3}{2c_M}M_{-1}Q_{-\frac{3}{2}}+\frac{1}{c_M}M_{-2}Q_{-\frac{1}{2}}+\frac{3}{2c_M^2}M_{-1}^{2}Q_{-\frac{1}{2}}.
		\end{eqnarray*}
	\end{exa}
	
	By direct calculation, we have the following  lemma.

	\begin{lem}\label{l3.15} For any $x\in \frak g_+$, we have
		\begin{eqnarray*}
			[x, {\rm R}]\subset U(\mathcal{M}_-+\mathcal{Q}_-)\left(2M_0+\frac{p^2-1}{12}{\bf c}_M\right)+U(\mathcal{M}_-+\mathcal{Q}_-)\frak g_+.
		\end{eqnarray*}
	\end{lem}
	\begin{proof} 
		It follows from the fact that $[x, {\rm R}]\1=0$ in $V(c_L,c_M,h_L,h_M)$ and $[x, {\rm R}]\in U(\mathcal{M}+\mathcal Q+\mathbb C{\bf c}_M)$ for any $x\in\frak g_+$.
	\end{proof}
	%
	%
	%
	%
	%
	%
	%

	\begin{pro}\label{t3.15} Let $p\in2\mathbb Z_+-1$. Then
		${\rm R}^{2}={\rm S}$, ${\rm R}^n\1$ is also a singular vector for any $n\in 2\mathbb Z_+-1$.
	\end{pro}
	\proof 
	It follows from Lemma \ref{l3.15} that ${\rm R}^{2}\1$ is a singular vector in $ V(c_L,c_M,h_L,h_M)$. By Theorem \ref{main1}, ${\rm R}^{2}={\rm S}$.
	Moreover, for any $n\in 2\mathbb Z_+-1$, ${\rm R}^n\1$ is also a singular vector by Lemma \ref{l3.15}.
	\qed

	\begin{lem}\label{l3.1Q}  If $u$ is a singular vector in $V(c_L,c_M,h_L,h_M)_{n-\frac12}$ for $n\ge 1$, then $\ell_{L}(u)=0$.
	\end{lem}
	
	\begin{proof} 	 Write $u=y\1\in V(c_L,c_M,h_L,h_M)_{n-\frac12}$, where $y\in U(\mathfrak{g}_-)$. Then $M_0 u=M_0y\1=[M_0,y]\1+h_Mu$. By a similar argument in the proof of Lemma \ref{l3.1}, we have $M_0y\1=h_My\1$.  For any $x\in \frak g_+$, $x{\rm R}y\1={\rm R}xy\1+[x, {\rm R}]y\1=0$ by Lemma \ref{l3.15}.
		Then ${\rm R}y\1$ is a singular vector in $V(c_L,c_M,h_L,h_M)_{n+\frac{p-1}2}$. So ${\ell}_L({\rm R}y)=0$ by Lemma \ref{l3.1} and then ${\ell}_L(y)=0$.
	\end{proof}

	Now we get the following result.
	
	\begin{theo}
		\label{main2} Let  $(c_L,c_M,h_L,h_M)\in\bC^4$ such that  $\phi(p)=0$  with $c_M\ne 0$. Then $V(c_L,c_M,h_L,h_M)_{n-\frac{1}{2}}$ for $n\in\mathbb Z_+$ has a singular vector $u$ if and only if $p\in 2\mathbb Z_+-1$ and there exists $k\in \mathbb Z_+$ such that $n-\frac12=\frac{p}{2}(2k-1)$. Moreover,
		all singular vectors of $V(c_L,c_M,h_L,h_M)_{kp-\frac{p}{2}}$, up to a scalar multiple, are ${\rm R}^{2k-1}{\bf 1}$ for $k\in \mathbb{Z}_+$.
	\end{theo}
	\proof 
	By Lemmas \ref{singular-Q1}, \ref{l3.1Q} and Propositions \ref{singular-R1} and \ref{t3.15},  we can suppose that
	\begin{equation}
		u= (a_0{\rm R}^{2k-1}+a_1{\rm R}^{2k-3}+\cdots a_{k-1}{\rm R}+a_k)\mathbf1,\label{singularSM}
	\end{equation}
	where $k\in\mathbb Z_+, a_i\in U(\mathcal{M}_-)$ not involving $M_{-p}$ for any $i=1, \cdots, k-1$, and $a_k\in U(\mathcal{M}_-+\mathcal{Q}_-)$ not involving $M_{-p}, Q_{-\frac p2}$.
	
	Assume that $a_k\ne 0$, then $\ell_Q(a_k)=1$. Set ${\rm hm}(a_k)=M_{-\mu}Q_{-\frac q2}$, where $\mu\in\mathcal P, q\ne p$. By action of $Q_{\frac q2}$ on $u$, we get a contradiction. So $a_k=0$.
	
	Set ${\rm max}\{{\rm hm}(a_0), \cdots, {\rm hm}(a_{k-1})\}=M_{-\lambda}$. By actions of $L_{\lambda}$ on \eqref{singularSM},  we can get 
	$L_{\lambda}u\ne0$ since all $a_i\in U(\mathcal M_-)$ are not involving $M_{-p}$.
	So $a_i\in\mathbb C$  for any $i=0,1, \cdots, k-1$.  The theorem follows.
	\qed

	Combining Theorem $\ref{main1}$ with Theorem $\ref{main2}$, we get the following result about all singular vectors of the Verma $\mathfrak{g}$-module $V(c_L,c_M,h_L,h_M)$.
	
	\begin{theo}\label{t3.19}  Let  $(c_L,c_M,h_L,h_M)\in\bC^4$ such that  $\phi(p)=0$  with $c_M\ne 0$.
		\begin{itemize}
			\item[$(1)$] If $p$ is even,  all singular vectors of the Verma module $V(c_L,c_M,h_L,h_M)$ are ${\rm S}^k{\bf 1}$ for   $k\in \mathbb N$, up to a scalar multiple.
			\item[$(2)$]   If $p$ is odd,  all singular vectors of the Verma module $V(c_L,c_M,h_L,h_M)$ are ${\rm R}^{k}{\bf 1}$ for   $k\in \mathbb N$, up to a scalar multiple.	\end{itemize}
		
	\end{theo}
	
	Applying this theorem we can easily get the following consequence.

	\begin{cor}
		Let $(c_L,c_M,h_L,h_M)\ne (c_L',c_M',h_L',h_M')\in\bC^4$.
		Then $${\rm Hom}_{\frak g} (V(c_L,c_M,h_L,h_M),  V(c_L',c_M',h_L',h_M'))\ne 0$$ if and only if $c_M=c_M', c_L=c_L', h_M=h_M'$, $2h'_M+\frac{p^2-1}{12}c'_M=0$ for some $p\in \mathbb Z_+$,  and $h_L=h_L'+ip$ for some $i\in \mathbb N$ (when $p$ even) or $i\in \frac12\mathbb N$ (when $p$ odd).
	\end{cor}
	\begin{proof}  We know that ${\rm Hom}_{\frak g} (V(c_L,c_M,h_L,h_M),  V(c_L',c_M',h_L',h_M'))\ne 0$ if and only if there is  a non-zero $\frak g$-module homomorphism
		$$\varphi: V(c_L,c_M,h_L,h_M)=\langle {\bf 1}\rangle\to V(c_L',c_M',h_L',h_M')=\langle {\bf1'}\rangle,$$ if and only if,   $\varphi({\bf 1})=u{\bf 1'}$ is a singular vector of $ V(c_L',c_M',h_L',h_M')$ for some $u\in U(\frak g_-) $, by Theorem \ref{t3.19}, if and only if   $u={\rm S}^k$ ($p$ even)  or ${\rm R}^k$   ($p$ odd) for some $k\in\mathbb N$. So $c_M=c_M', c_L=c_L', h_M=h_M'$ and $h_L=h_L'+ip$ for some $i\in \mathbb N$ (when $p$ even) or $i\in \frac12\mathbb N$ (when $p$ odd). In this case ${\rm dim}\, {\rm Hom}_{\frak g}(V(c_L,c_M,h_L,h_M), V(c_L',c_M',h_L',h_M'))=1$.
	\end{proof}

	\begin{cor}
		\label{main1-w22} Using the notations as above, if   $(c_L,c_M,h_L,h_M)\in\bC^4$ such that  $\phi(p)=0$  with $c_M\ne 0$,  then any singular vector  of the Verma module  $V_{W(2,2)}(h_L, h_M, c_L, c_M)$ is  ${\rm S}^k{\bf 1}$ for some  $k\in \mathbb{N}$, up to a scalar multiple.
	\end{cor} 
	
	\proof Consider the subspace $U({W(2,2)})\1$ in the Verma $\mathfrak{g}$-module $V(h_L, h_M, c_L, c_M)$ which is the Verma $W(2,2)$-module $V_{W(2,2)}(h_L, h_M, c_L, c_M)$. From Corollary \ref{t3.19} and simple computations we know that $u\in V_{W(2,2)}(h_L, h_M, c_L, c_M)$ is a singular vector if and only if it is a singular vector in  the Verma $\mathfrak{g}$-module $V(c_L,c_M,h_L,h_M)$,   if and only if it is ${\rm S}^k{\bf 1}$ for   $k\in \mathbb{N}$, up to a scalar multiple.
	\qed
	
	\begin{rem}
 Corollary \ref{main1-w22} was originally proposed in \cite[Theorem 2.7]{JZ}. However, the proof presented in \cite[Lemma 2.4, Theorem 2.7]{JZ} contains certain gaps. The singular vector ${\rm S}\1$ for the Verma module $V_{W(2,2)}(h_L, h_M, c_L, c_M)$ was first introduced in \cite[Proposition 2.6]{JZ}, and later expressed in \cite[Theorem 7.5]{AR1} using a free-field realization of vertex algebras and Schur polynomials.

	\end{rem}

	\section{Classification of subsingular vectors of Verma modules}

	In this section, we continue considering reducible Verma modules $V(c_L,c_M,h_L,h_M)$ over $\mathfrak{g}$ for fixed 
	$(c_L,c_M,h_L,h_M)\in\bC^4$.
	So we always assume that  $\phi(p)=2h_M+\frac{p^2-1}{12}c_M=0$ for some $p\in \mathbb Z_+$ with $c_M\neq 0$. We will determine all subsingular vectors   in the Verma module $V(h_L, h_M, c_L, c_M)$.

	Let $J'(c_L,c_M,h_L,h_M)$ be the submodule of $V(c_L,c_M,h_L,h_M)$ generated by all singular vectors.
	Set
	$$
	L'(c_L,c_M,h_L,h_M)=V(c_L,c_M,h_L,h_M)/J'(c_L,c_M,h_L,h_M).
	$$
	By Theorem \ref{t3.19},	   $J'(c_L,c_M,h_L,h_M)$ is generated by $u={\rm S}\1$ if $p\in 2\mathbb Z_+$,   by $u={\rm R}\1$ if $p\in 2\mathbb Z_+-1$, defined in Section 3.

	For convenience, for  $x\in V(c_L,c_M,h_L,h_M)$ we will abuse the notation that $ x \in L'(c_L,c_M,h_L,h_M)$ means $ x+J'(c_L,c_M,h_L,h_M)\in L'(c_L,c_M,h_L,h_M)$. 
	
	\subsection{Necessary condition for the existence of subsingular vectors}

	From the construction of ${\rm R}$ and ${\rm S}$ we have the following results.
	\begin{lem}\label{ll4.1} 
		(1) If $p\in 2\mathbb Z_+$, then the image of
		\begin{eqnarray}\label{e4.1}
			{\mathcal B}=\{M_{-\la}Q_{-\mu}L_{-\nu}{\bf 1}\mid   \la,\nu\in \mathcal P, \mu\in\mathcal{SP}, \ \mbox{and}\ M_{-\la}\ \mbox{does't involve }\ M_{-p}\}
		\end{eqnarray}
		under the natural projection $$\pi: V(c_L,c_M,h_L,h_M)\rightarrow L'(c_L,c_M,h_L,h_M)=V(c_L,c_M,h_L,h_M)/J'(c_L,c_M,h_L,h_M)$$ forms a PBW basis of $L'(c_L,c_M,h_L,h_M)$.\\
		(2) If $p\in 2\mathbb Z_+-1$, then the image of
		\begin{equation}\label{e4.2}
			{\mathcal B}'=\{M_{-\la}Q_{-\mu}L_{-\nu}{\bf 1}\mid    \la,\nu\in \mathcal P, \mu\in\mathcal{SP}, \ \mbox{and}\ \  Q_{-\mu},M_{-\la}\ \mbox{does't involve }\ Q_{-\frac{p}{2}},M_{-p}
			\ \mbox{respectively}\}
		\end{equation}
		under the natural projection $\pi$ forms a PBW basis of $L'(c_L,c_M,h_L,h_M)$.
	\end{lem}

	\begin{lem}\label{hmsubsingular}
		If $L'(c_L,c_M,h_L,h_M)$ is reducible and $u'$ is a singular vector not in $\mathbb C\1$, then  ${\rm hm}(u')=L_{-p}^{r}{\bf 1}$ for some $r\in \mathbb Z_+$, and $\ell_{L}(u')=r$.
	\end{lem}
	\proof  By Lemma \ref{ll4.1}, we may assume that  any term of $u'$ does not involve $M_{-p}$ or $Q_{-\frac{p}{2}}$ (this factor does not appear if $p$ is even).
	
	If $\ell_{L}(u')=0$, using similar discussions in Section 3 (see the beginning part of the proof of Lemma \ref{l3.6}, and Theorem \ref{singular-W}),  we can get $u'\in J'(c_L,c_M,h_L,h_M)$, a contradiction. 
	
	So $\ell_{L}(u')\ne0$, and suppose that
	\begin{equation*}
		u'= (g_0L_{-p}^r+g_1L_{-p}^{r-1}+g_2L_{-p}^{r-2}+\cdots+g_r)\1, \label{subsingularL}
	\end{equation*}
	where $r\in\mathbb Z_+, g_i\in U(\frak g_-), i=0, 1, \cdots, r,  g_0\ne 0$ and any $g_i$ does not involve $L_{-p}, M_{-p}, Q_{-\frac p2}$.
	
	By the proof of Lemma \ref{l3.1} (ignoring the eigenvalue of $u'$), we can get ${\ell}_L(g_0)=0$. 
	Using the proof of Lemma \ref{l3.6}, we have ${\ell}_Q(g_0)=0$. So $g_0\in U(\mathcal{M}_-)$. Now we need to show that   $g_0\in \mathbb C$.
	
	(1) First we consider the case of $p=1$. Note that $h_M=0$, hence $[L_{-1},M_1]=0$. If $\ell_L(g_1)\ne 0$.  Set ${\rm hm}(g_1)=b(M, Q)L_{-\nu}$ for some $b(M, Q)\in U(\mathcal{M}_-+\mathcal Q_-)$. Then $\nu_1>1$.
	By the action of $M_{\nu_1}$ on $u'$, we can get a contradiction by comparing the coefficient of $L_{-1}^{r-1}$. So $\ell_L(g_1)=0$. Similarly, we have $\ell_L(g_2)=\cdots=\ell_L(g_{r})=0$ since $M_0L_{-1}^j\1=0, M_kL_{-1}^j\1=0$ for any $k, j\in\mathbb Z_+$ (Theorem \ref{cor3.3}).
	
	If $g_0\notin \mathbb C$, set ${\rm hm}\,(g_0)=M_{-\mu}$ not involving $M_{-1}$, then
	$$L_{\mu_1}u'=[L_{\mu_1}, g_0]L_{-1}^r+B, $$ where the degree of $L_{-1}$ in $B$ is no more than $r-1$ and $ [L_{\mu_1}, M_{-\mu}]\1\ne 0$.  It gets a contradiction.
	So   $g_0\in\mathbb C^*$.  Consequently, ${\rm hm}(u')=L_{-1}^{r}{\bf 1}$. In this case, $g_1=0$ since $g_1$ is not involving $M_{-1}, Q_{-\frac12}$.

	(2) Now we consider the case of $p>1$.
	
	As in Lemma \ref{l3.1} and Lemma \ref{l3.6}  (using $M_1$ instead of $M_0$ in the arguments), we get 
	\begin{equation*}\ell_L (g_1)=1\  {\rm and}\   g_1=\sum_{i=1}^{s}b_iL_{-i}+b_0,\label{g1}\end{equation*} where $b_i\in, i=1, \cdots, s$,  $b_s\ne 0$ and $s<p$, $b_0\in \mathcal{MQ}$.
	
	Moreover, we can get
	\begin{eqnarray*}
		\ell_L (g_i)=i
	\end{eqnarray*}
	for $i=1,\cdots,r$ by induction, and all $L_{-\nu}$ in $g_i, i\ge1$ must be satisfied the condition that $\nu_1<p$ (see the proof of Lemma \ref{l3.1} using $M_1$ instead $M_0$ in the arguments).
	In this case $\ell_{L}(u')=r$.
	
	Now we shall prove that $g_0\in \mathbb C^*$. Otherwise, set ${\rm hm}\,(g_0)=M_{-\mu}$ not involving $M_{-p}$, then
	$$L_{\mu_1}u'=[L_{\mu_1}, g_0]L_{-p}^r+B, $$ where the degree of $L_{-p}$ in $B$ is no more than $r-1$ and $ [L_{\mu_1}, M_{-\mu}]\1\ne 0$.  It gets a contradiction.
	The lemma follows.
	\qed

	%
	%
	
	Lemma \ref{hmsubsingular} tells us that, if there exists a singular vector $u'\in L'(c_L,c_M,h_L,h_M)$ of weight $pr$, then it is unique up to a scalar multiple.
	In the following, we provide the necessary conditions for the existence of subsingular vectors in the Verma module $V(h_L, h_M, c_L, c_M)$ over the N=1 BMS superalgebra $\mathfrak g$.
	\begin{theo}\label{necessity} Let  $(c_L,c_M,h_L,h_M)\in\bC^4$ such that    $\phi(p)=2h_M+\frac{p^2-1}{12}c_M=0$ for some $p\in \mathbb Z_+$ and $c_M\neq 0$. Assume that there exists a singular vector $u'\in L'(c_L,c_M,h_L,h_M)$
		such that ${\rm hm}(u')=L_{-p}^{r}\1$ for some $r\in \mathbb Z_+$.
		Then $h_L=h_{p, r}$ where 
		\begin{eqnarray}\label{e3.37}\label{atypical}
			h_{p,r}=-\frac{p^2-1}{24}c_L+\frac{(41p+5)(p-1)}{48}+\frac{(1-r)p}{2}-\frac{1+(-1)^p}8p.
		\end{eqnarray}
	\end{theo}
	\proof
	{\bf Case 1}: $p=1$.  From the proof of Lemma \ref{hmsubsingular}, we can suppose that 
	$$u'=(L_{-1}^r+g_2L_{-1}^{r-2}+\cdots+g_{r-1}L_{-1}+g_{r})\1, $$ where $r\in\mathbb Z_+$, each $g_i\in  U(\mathcal{M}_-+\mathcal{Q}_-)$  does not involve $M_{-1}, Q_{-\frac 12}$ for $ i=1,2, \cdots, r$.
	Considering the coefficient of $L_{-1}^{r-1}$ in $L_1u'$
	and using the formula
	$$L_1L_{-1}^r \1=L_{-1}^{r-1}\left(rL_0+\frac{r(r-1)}2\right)\1
	$$
	we can get $h_L=\frac{1-r}2$ by comparing the coefficient of $L_{-1}^{r-1}$.
	
	{\bf Case 2}: $p>1$. 
	From Lemma \ref{hmsubsingular}, we can suppose that 
	$$u'=(L_{-p}^r+g_1L_{-p}^{r-1}+\cdots+g_{r-1}L_{-p}+g_{r})\1, $$ where $r\in\mathbb Z_+$, $g_i\in U(\frak g_-), i=1,2, \cdots, r$ do not involve $L_{-p}, M_{-p}, Q_{-\frac p2}$.
	By Lemma \ref{hmsubsingular}, we can further assume that  \begin{equation}g_1=\sum_{i=1}^{p-1}l_{i}M_{-i}L_{-(p-i)}+\sum_{j=1}^{\lfloor \frac{p}{2}\rfloor}n_iQ_{-p+i-\frac{1}{2}}Q_{-i+\frac{1}{2}}+C,\label{g1-exp}\end{equation}
	where $l_i, n_j\in\mathbb C$ and 
	\begin{equation}C=\sum_{\stackrel{i=1, 2, \cdots, p-1}{\ell(\la)\ge 2}}a_{\la, i}M_{-\la}L_{-i}+\sum_{\ell(\la)+\ell(\mu)\ge 3}b_{\la, \mu}M_{-\la}Q_{-\mu+\frac12}\label{g1-C}\end{equation} for some $a_{\mu, i}, b_{\la, \mu}\in\mathbb C$.
	

	For any $k=1,2,\cdots, p-1$, 
	\begin{eqnarray*}\label{Lkaction}
		L_ku'&=&[L_k, L_{-p}^r+g_1L_{-p}^{r-1}+\cdots+g_{r-1}L_{-p}+g_{r}]\mathbf 1\\
		&=&([L_k, L_{-p}^r]+[L_k, g_1]L_{-p}^{r-1}+B)\1,
	\end{eqnarray*}
	where the degree of $L_{-p}$ in $B$ is less than $r-2$.
	The coefficient with $L_{-p}^{r-1}\1$ in $L_{k}u'$ should be zero.
	
	Comparing the coefficients of $L_{-p+k}L_{-p}^{r-1}$ in $L_{k}u'$, we can get 
	$r(k+p)+l_k(2kh_M+\frac{k^3-k}{12}c_M)=0$, yielding that 
	
	\begin{eqnarray*}
		l_k=-r\frac{p^2-1}{2h_Mk(p-k)}
	\end{eqnarray*}
	for $k=1,\ldots,p-1$. Note that here the degree of $L_{-p}$ of $[L_k, C]L_{-p}^{r-1}\1$ is $r-1$, or $r-2$. For the former,  the length of  any non-zero summand
	in $[L_k, C]$ is not less than $2$ with respect to $M$ (see \eqref{g1-C}).
	
	For any $k=1,2,\cdots, \lfloor \frac{p}{2}\rfloor$, comparing the coefficients of $Q_{-p+k-\frac 12}L_{-p}^{r-1}$ in $Q_{k-\frac 12}u'$, we obtain that 
	$ \frac{p+2k-1}2+\left(2h_M-\frac{8(k^2-k)h_M}{p^2-1}\right)n_k=0$, yielding that 
	\begin{eqnarray*}
		n_{k}=r\frac{p^{2}-1}{4h_M(p-2k+1)}.
	\end{eqnarray*}
	
	Note that 
	\begin{eqnarray*}
		[L_{p}, L_{-p}^{r}]=rpL_{-p}^{r-1}\Big((r-1)p+2L_{0}+\frac{p^{2}-1}{12}c_L\Big).
	\end{eqnarray*}
	The coefficient with $L_{-p}^{r-1}\1$ in $L_{p}u'$ is
	\begin{eqnarray}\label{e3.401}
		rp\Big((r-1)p+2h_L+\frac{p^{2}-1}{12}c_L\Big)&+&\sum_{i=1}^{p-1}2l_{i}h_Mi(2p-i)\frac{p^{2}-i^{2}}{p^{2}-1}\notag\\
		&+&\sum_{i=1}^{\lfloor \frac{p}{2}\rfloor}n_{i}h_M(3p-2i+1)\frac{p^{2}-(1-2i)^{2}}{p^{2}-1}=0,
	\end{eqnarray}
	i.e.,
	\begin{eqnarray*}
		&&rp\Big((r-1)p+2h_L+\frac{p^{2}-1}{12}c_L\Big)-2rp^2(p-1)-\frac{rp(p^2-1)}6 +\frac14 r(p-1)(3p+1)\left\lfloor \frac{p}{2}\right\rfloor\notag\\
		+&&\frac12 r(p+1)\left\lfloor \frac{p}{2}\right\rfloor(\left\lfloor \frac{p}{2}\right\rfloor+1)-\frac r6\left\lfloor \frac{p}{2}\right\rfloor(\left\lfloor \frac{p}{2}\right\rfloor+1)(2\left\lfloor \frac{p}{2}\right\rfloor+1)=0.\label{e3.40}
	\end{eqnarray*}
	It gives \eqref{atypical}.
	\qed
	
	In the above proof we did not use the actions of $M_k$ for $1\le k<p$ because they can be generated by $Q_{i-\frac12}$ (for example $M_1=Q_{\frac12}^2$). This tells us that, for $u'$, the  summands $\sum_{i=1}^{p-1}l_{i}M_{-i}L_{-(p-i)}+\sum_{j=1}^{\left\lfloor \frac{p}{2}\right\rfloor}n_iQ_{-p+i-\frac{1}{2}}Q_{-i+\frac{1}{2}}$ in \eqref{g1-exp} are unique determined. We will particularly use this fact for $r=1$ later.
	
	Now we first determine singular vectors in $L'(c_L,c_M,h_L,h_M)_p$ under the condition $h_L=h_{p, 1}$.
	
	\begin{lem}\label{l4.4}  If $u'$ is a singular vector in $L'(c_L,c_M,h_L,h_M)_p$ (implying that $h_L=h_{p, 1}$), then $u'$ can be written as follows.
		\begin{eqnarray}\label{subsingular2}
			u'={\rm T}\1=L_{-p}\1+\sum_{i=1}^{p-1}g_{p-i}(M)L_{-i}\1+u_p(M,Q)\1,
		\end{eqnarray}
		where $g_{i}(M)\in  U(\mathcal{M}_-)$, $u_{p}(M,Q)\in  U(\mathcal{M}_-+\mathcal{Q}_-)$ not involving   $M_{-p}$ or $Q_{-\frac{p}{2}}$, and $\ell_Q(u_{p}(M,Q))= 2$.
	\end{lem}
	\proof The case for $p=1$ is clear since $h_{1,1}=0$ and $u'=L_{-1}$. So we need only to consider  the case for $p>1$.  By Lemma \ref{hmsubsingular}, we may assume that   $u'={\rm T}\1$ where
	\begin{eqnarray}
		{\rm T}=L_{-p}+\sum_{i=1}^{p-1}g_{p-i}(M,Q)L_{-i} +u_p(M,Q),
	\end{eqnarray}
	and $g_{i}(M,Q)\in  U(\mathcal{M}_-+\mathcal{Q}_-)$ not involving  $M_{-p},Q_{-\frac{p}{2}}$. Note that $M_0u'=ku'$ with $k\in \mathbb{C}$. On one hand, $[M_0,{\rm T}]=pM_{-p}+\sum_{i=1}^{p-1}ig_{p-i}(M,Q)M_{-i}$. Then $[M_0,{\rm T}]\neq k'{\rm T}$ for any $k'\in \mathbb{C}$. So $[M_0,{\rm T}]\1\in J'(c_L,c_M,h_L,h_M)_p$. It implies $[M_0,{\rm T}]=l{\rm S}$ for some $l\in \mathbb{C}^*$. On the other hand, ${\rm S}=M_{-p}+f_p(M)$. So $l=p$ and in $U(\mathfrak{g_{-}})$, we get
	\begin{equation}
		[M_0,{\rm T}]=p{\rm S}.\label{W0T}  \end{equation} This implies $\sum_{i=1}^{p-1}ig_{p-i}(M,Q)M_{-i}=p{\rm S}$. So  $g_{p-i}(M,Q)\in U(\mathcal M_-)$. We denote it by $g_{p-i}(M)$ for any $i=1,2, \cdots, p-1$. 
	
	For $1\le k\le p,$ considering $$0=Q_{k-\frac12}u'=\left(\frac p2+k-\frac12\right)Q_{-p+k-\frac12}\1+\sum_{i=1}^{p-1}\left(\frac i2+k-\frac12\right)g_{p-i}(M)Q_{-i+k-\frac12}\1+[Q_{k-\frac12},u_p(M,Q)]\1,
	$$  we see that  $\ell_Q(u_{p}(M,Q))=2$.
	This completes the proof.
	\qed
	
	
	%
	\begin{rem}
		We found the element ${\rm T}\in U(\frak{g}_-)$ when $h_L=h_{p,1}$. From the above proof we know that  (\ref{W0T})
		holds whenever $\phi(p)=0$, no need to assume that $h_L=h_{p, 1}$. 
	\end{rem}
	
	\begin{theo}\label{subsingular} Let  $(c_L,c_M,h_L,h_M)\in\bC^4$ such that   $\phi(p)=0$ and $h_L=h_{p, 1}$ for some $p\in \mathbb{Z_+}$. Then there exists a unique subsingular vector $u'={\rm T}{\bf 1}$ in $L'(c_L,c_M,h_L,h_M)_p$  up to a scalar multiple, where ${\rm T}$ is defined in Lemma \ref{l4.4}. 
	\end{theo}
	\begin{proof} By Lemma \ref{l4.4}, we can suppose that
		\begin{eqnarray}\label{subsingular2'}
			u'={\rm T}\1=L_{-p}\1+\sum_{i=1}^{p-1}g_{i}(M)L_{-p+i}\1+u_p(M,Q)\1,
		\end{eqnarray}
		where $g_{i}(M)\in U(\mathcal{M}_-)$, $u_{p}(M,Q)\in  U(\mathcal{M}_-+\mathcal{Q}_-)$ not involving  $M_{-p},Q_{-\frac{p}{2}}$,  and $\ell_Q(u_p(M, Q))=2$.

		We order all the possible summands of ${\rm T}$ in \eqref{subsingular2}
		by the ordering $\succ$ defined in Section 2.2:
		\begin{eqnarray} \nonumber
			&&L_{-p}, M_{-1}L_{-(p-1)}, M_{-2} L_{-(p-2)}, M_{-1}^2L_{-(p-2)}, \cdots, M_{-(p-1)}L_{-1}, \cdots,
			M_{-1}^{p-1}L_{-1},\\ \nonumber
			&&Q_{-p+\frac{1}{2}}Q_{-\frac{1}{2}}, Q_{-p+\frac{3}{2}}Q_{-\frac{3}{2}}, 
			\cdots,Q_{-p+\lfloor \frac{p}{2}\rfloor -\frac{1}{2}}Q_{-\lfloor \frac{p}{2}\rfloor +\frac{1}{2}}, \\ \nonumber
			&&M_{-1}Q_{-p+\frac{3}{2}}Q_{-\frac{1}{2}}, M_{-2}Q_{-p+\frac{5}{2}}Q_{-\frac{1}{2}}, M_{-1}^2Q_{-p+\frac{5}{2}}Q_{-\frac{1}{2}},\cdots, M_{-p+2}Q_{-\frac{3}{2}}Q_{-\frac{1}{2}},  M_{-1}^{p-2}Q_{-\frac{3}{2}}Q_{-\frac{1}{2}},\\
			&&M_{-(p-1)}M_{-1}, M_{-(p-2)}M_{-2}, M_{-(p-2)}M_{-1}^2,\cdots, M_{-1}^{p}. \label{singu-order}
		\end{eqnarray}
		The coefficients of above monomials in $u'$ are determined by some elements in $U(\mathfrak{g}_{+})_{-p}$ which act on $u'$ getting $0$. 
		Namely, we need to consider the linear equations
		\begin{equation}xu'=0\label{singular-equation}\end{equation}  for some particular $x\in U(\frak g_+)_{-p}$.
		
		We  choose $x$ from \eqref{singu-order} by changing  $L_{-p}$ to $L_p$, $L_{-i}$ to $M_i$, $M_{-i}$ to $L_i, i=1, \cdots p-1$, $Q_{-r}$ to $Q_r$, and arrange them according to the original ordering as follows:
		$L_{p}$, $L_{1}M_{p-1}$, $L_{2}M_{p-2}$, $L_{1}^2M_{p-2}$, $\cdots$, $L_{p-1}M_{1}$, $\cdots$, $L_{-1}^{p-1}M_{1}$, $Q_{\frac{1}{2}}Q_{p-\frac{1}{2}}$, $Q_{\frac{3}{2}}Q_{p-\frac{3}{2}}$, $\cdots$, $Q_{\lfloor \frac{p}{2}\rfloor -\frac{1}{2}}Q_{p-\lfloor \frac{p}{2}\rfloor +\frac{1}{2}}$,
		$L_1Q_{\frac{1}{2}}Q_{p-\frac{3}{2}},\cdots, L_{1}^{p-2}Q_{\frac{1}{2}}Q_{\frac{3}{2}}$,
		$L_{1}L_{p-1}$, $L_{2}L_{p-2}$, $L_{1}^2L_{p-2}$, $\cdots$, $L_{1}^p$. We consider the following Table 1.
		
		\setlength{\belowcaptionskip}{-10pt}		
		\begin{table}[htbp]\label{table1}
			\centering\caption{The matrix $A_{p, 1}$}
			\begin{eqnarray*}\label{sub-table} \fontsize{4.98pt}{\baselineskip}\selectfont 
				\begin{tabular}{|c|c|c|c|c|c|c|c|c|c|c|c|c|c|}\hline 
					&$L_{-p}{\bf 1}$  & $M_{-1} L_{-(p-1)}{\bf 1}$ & $\cdots$ & $M_{-1}^{p-1}L_{-1}{\bf 1}$ &$Q_{-p+\frac{1}{2}}Q_{-\frac{1}{2}}{\bf 1}$ & $\cdots$ &$Q_{-p+\lfloor \frac{p}{2}\rfloor -\frac{1}{2}}Q_{-\lfloor \frac{p}{2}\rfloor +\frac{1}{2}}{\bf 1}$ & $M_{-1}Q_{-p+\frac{3}{2}}Q_{-\frac{1}{2}}{\bf 1}$ & $\cdots$   & $M_{-1}^{p-2}Q_{-\frac{3}{2}}Q_{-\frac{1}{2}}{\bf 1}$   & $M_{-(p-1)}M_{-1}{\bf 1}$   & $\cdots$   & $M_{-1}^p{\bf 1}$  \\ \hline
					$L_{p}$  & $\cellcolor{gray!50}\star$ & $\cellcolor{gray!50}\star$ & $\cellcolor{gray!50}\cdots$ & $\cellcolor{gray!50}\star$ & $\cellcolor{gray!50}\star$ & $\cellcolor{gray!50}\cdots$ &  $\cellcolor{gray!50}\star$ & $0$ & $\cdots$   & $0$ & $0$ & $0$ & $0$  \\ \hline
					$L_{1}M_{p-1}$  &  $\cellcolor{gray!50}\star$ & $\cellcolor{gray!50}\star$ & $\cellcolor{gray!50}\cdots$   & $\cellcolor{gray!50}\star$ & $\cellcolor{gray!50}0$ & $\cellcolor{gray!50}\cdots$ & $\cellcolor{gray!50}0$  & $0$& $\cdots$   & $0$& $0$& $0$& $0$ \\ \hline
					$\vdots$  & $ \cellcolor{gray!50}\vdots$& $\cellcolor{gray!50}\vdots$& $\cellcolor{gray!50}\vdots$   & $\cellcolor{gray!50}\vdots$& $\cellcolor{gray!50}\vdots$& $\cellcolor{gray!50}\vdots$& $\cellcolor{gray!50}\vdots$  & $0$& $\cdots$   & $0$& $0$& $0$& $0$  \\ \hline
					$L_{1}^{p-1}M_{1}$    & $ \cellcolor{gray!50}\star$& $\cellcolor{gray!50}\star$& $\cellcolor{gray!50}\cdots$   & $\cellcolor{gray!50}\star$& $\cellcolor{gray!50}0$& $\cellcolor{gray!50}\cdots$& $\cellcolor{gray!50}0$  & $0$& $\cdots$   & $0$& $0$& $0$& $0$   \\ \hline
					$Q_{p-\frac{1}{2}}Q_{\frac{1}{2}}$    & $ \cellcolor{gray!50}\star$& $\cellcolor{gray!50}\star$& $\cellcolor{gray!50}\cdots$   & $\cellcolor{gray!50}\star$& $\cellcolor{gray!50}\star$& $\cellcolor{gray!50}\cdots$& $\cellcolor{gray!50}\star$  & $0$& $\cdots$   & $0$& $0$& $0$& $0$  \\ \hline
					$\vdots$   & $ \cellcolor{gray!50}\vdots$& $\cellcolor{gray!50}\vdots$& $\cellcolor{gray!50}\vdots$   & $\cellcolor{gray!50}\vdots$& $\cellcolor{gray!50}\vdots$& $\cellcolor{gray!50}\vdots$& $\cellcolor{gray!50}\vdots$  & $0$& $\cdots$   & $0$& $0$& $0$& $0$  \\ \hline
					$Q_{p-\lfloor \frac{p}{2}\rfloor +\frac{1}{2}}Q_{\lfloor \frac{p}{2}\rfloor -\frac{1}{2}}$  & $ \cellcolor{gray!50}\star$& $\cellcolor{gray!50}\star$& $\cellcolor{gray!50}\cdots$   & $\cellcolor{gray!50}\star$& $\cellcolor{gray!50}\star$& $\cellcolor{gray!50}\cdots$& $\cellcolor{gray!50}\star$  &  $0$& $\cdots$   & $0$& $0$& $0$& $0$ \\ \hline
					$L_1Q_{p-\frac{3}{2}}Q_{\frac{1}{2}}$ & $\star$  & $\star$  & $\star$  & $\star$ & $\star$ & $\star$ & $\star$  & $\cellcolor{gray!50}\star$ & $\cellcolor{gray!50} 0$ & $\cellcolor{gray!50} 0$   & $0$   & $0$   & $0$  \\ \hline
					$\vdots$ & $\vdots$  & $\vdots$ & $\vdots$ & $\vdots$ & $\vdots$ & $\vdots$  & $\vdots$  & $\cellcolor{gray!50}\star$& $\cellcolor{gray!50}\star$ & $\cellcolor{gray!50} 0$   & $0$   & $0$   & $0$  \\ \hline
					$L_1^{p-2}Q_{\frac{3}{2}}Q_{\frac{1}{2}}$ & $\star$  & $\star$  & $\star$  & $\star$ & $\star$ & $\star$ & $\star$  & $\cellcolor{gray!50}\star$  & $\cellcolor{gray!50}\star$ & $\cellcolor{gray!50}\star$ & $0$ & $0$ & $0$  \\ \hline
					$L_{p-1}L_1$ & $\star$  & $\star$  & $\star$  & $\star$ & $\star$ & $\star$ & $\star$   & $\star$  & $\star$   & $\star$ & $\cellcolor{gray!50}\star$   & $\cellcolor{gray!50} 0$& $\cellcolor{gray!50} 0$  \\ \hline
					$\vdots$  & $\vdots$  & $\vdots$  & $\vdots$  & $\vdots$  & $\vdots$  & $\vdots$  & $\vdots$  & $\star$  & $\star$ & $\star$ & $\cellcolor{gray!50}\star$ & $\cellcolor{gray!50}\star$  & $\cellcolor{gray!50} 0$  \\ \hline
					$L_1^p$ & $\star$  & $\star$  & $\star$  & $\star$ & $\star$ & $\star$ & $\star$   & $\star$ & $\star$ & $\star$ & $\cellcolor{gray!50}\star$ & $\cellcolor{gray!50}\star$   & $\cellcolor{gray!50}\star$  \\ \hline
				\end{tabular},
			\end{eqnarray*}
		\end{table}
		
		The $(i, j)$-entry in Table 1 is the coefficient of ${\bf 1}$ produced by the   $i$-th operator from Column $0$ acting on the monomial of the $j$-th element on Row $0$. 
		
		Now we shall investigate  the coefficient matrix $A_{p,1}$ of the linear equations \eqref{singular-equation} by using Table 1. This matrix $A_{p,1}$ is a lower trianglar block matrix. 
		Note that the lower two shaded submatrices in Table 1 are nonsingular lower triangular matrices (with nonzero diagonal entries). So we need only to consider the upper-left shaded submatrix which will be denoted by $A_p$. In addition, these operators ($L_{p}$, $L_{1}M_{p-1}$, $\cdots$, $Q_{\lfloor \frac{p}{2}\rfloor -\frac{1}{2}}Q_{p-\lfloor \frac{p}{2}\rfloor +\frac{1}{2}}$) from Column $0$ except for $L_{\la}Q_{\mu-\frac12}$ with $\ell(\la)\ge 2$ in Table 1 act trivially on the monomial $M_{-\la}Q_{-\mu+\frac12}$ with $\ell(\la)\ge 2$ respectively. In order to calculate the rank of matrix $A_p$  we only need to consider a better submatrix $B_p$ of the matrix $A_p$ as Table 2. Actually,  after row and column operations, $A_p$ can be arranged as a lower block-triangular matrix with $B_p$ to be the upper-left block with corank$(A_p)=$corank$(B_p)$. It is clear that corank$(B_p)=0$ or $1$.
		
		\setlength{\belowcaptionskip}{-10pt}
		\begin{table}[htbp]
			\centering\caption{The matrix $B_p$}\label{table 2} 
			\begin{eqnarray*}\tiny\label{sub-table}
				\begin{tabular}
					{|c|c|c|c|c|c|c|c|c|}\hline
					& $L_{-p}{\bf 1}$  & $M_{-1} L_{-(p-1)}{\bf 1}$ & $M_{-2} L_{-(p-2)}{\bf 1}$ & $\cdots$ & $M_{-(p-1)} L_{-1}{\bf 1}$ & $Q_{-p+\frac{1}{2}}Q_{-\frac{1}{2}}{\bf 1}$ & $\cdots$ & $Q_{-p+\lfloor \frac{p}{2}\rfloor -\frac{1}{2}}Q_{-\lfloor \frac{p}{2}\rfloor +\frac{1}{2}}{\bf 1}$ \\ \hline
					$L_{p}$     & $\cellcolor{gray!50}\star$  & $\cellcolor{gray!50}\star$   & $\cellcolor{gray!50}\star$   & $\cellcolor{gray!50}\cdots$ & $\cellcolor{gray!50}\star$  & $\cellcolor{gray!50}\star$   & $\cellcolor{gray!50}\cdots$ & $\cellcolor{gray!50}\star$     \\ \hline
					$L_{1}M_{p-1}$    & $\cellcolor{gray!50}\star$  & $\cellcolor{gray!50}\star$   & $0$  & $0$ & $0$   & $0$   & $0$  & $0$   \\ \hline
					$L_{2}M_{p-2}$   & $\cellcolor{gray!50}\star$  & $0$  & $\cellcolor{gray!50}\star$  & $0$  & $0$  & $0$   & $0$   & $0$  \\ \hline
					$\vdots$   & $\cellcolor{gray!50}\vdots$ & $0$   & $0$  & $\cellcolor{gray!50}\ddots$ & $0$ & $0$   & $0$   & $0$  \\ \hline
					$L_{p-1}M_{1}$   & $\cellcolor{gray!50}\star$  & $0$   & $0$ & $0$   & $\cellcolor{gray!50}\star$  & $0$  & $0$ & $0$  \\ \hline
					$Q_{p-\frac{1}{2}}Q_{\frac{1}{2}}$  & $\cellcolor{gray!50}\star$  & $0$   & $0$  & $0$  & $0$  & $\cellcolor{gray!50}\star$  & $0$  & $0$   \\ \hline
					$\vdots$   & $\cellcolor{gray!50}\vdots$ & $0$  & $0$  & $0$  & $0$  & $0$  & $\cellcolor{gray!50}\ddots$  & $0$  \\\hline
					$Q_{p-\lfloor \frac{p}{2}\rfloor +\frac{1}{2}}Q_{\lfloor \frac{p}{2}\rfloor -\frac{1}{2}}$ & $\cellcolor{gray!50}\star$  & $0$   & $0$   & $0$  & $0$ & $0$   & $0$   & $\cellcolor{gray!50}\star$    \\ \hline
				\end{tabular}.
			\end{eqnarray*}
		\end{table}

		%
		%
		%

		From the proof of Theorem \ref{necessity} with $r=1$, we know that the matrix $ B_p$ is of corank $1$ if and only if   $h_L=h_{p,1}$, that is,  the matrix $A_{p,1}$ is of corank $1$ if and only if   $h_L=h_{p,1}$, in which case there is only one singular vector $u'$ in $L'(c_L,c_M,h_L,h_M)_p$, up to a scalar multiple. 
	\end{proof}
	
	From th proof of Theorem \ref{necessity} we see that  that 
	\begin{equation*}\label{T-exp'}
		{\rm T}=L_{-p}+\sum_{i=1}^{p-1} \frac{12}{i(p-i)c_M} M_{-p+i}L_{-i}-\sum_{i=1}^{\lfloor \frac{p}{2}\rfloor }\frac{6}{(p-2k+1)c_M}Q_{i-p-\frac{1}{2}}Q_{-i+\frac{1}{2}}+\text{some other terms}.
	\end{equation*}
	We further have the following formula for {\rm T}.

	\begin{cor}\label{subsingular-T}   Let  $(c_L,c_M,h_L,h_M)\in\bC^4$ such that    $\phi(p)=2h_M+\frac{p^2-1}{12}c_M=0$ for some $p\in \mathbb Z_+$, $c_M\neq 0$ and $h_L=h_{p, 1}$. Let $k_i=\frac{12}{i(p-i)c_M},\ i=1, 2,\cdots, p-1$.
		Then the subsingular vector ${\rm T}\1$ can be determined as follows:
		\begin{equation}\label{T-exp}
			{\rm T}=L_{-p}+\sum_{i=1}^{p-1}g_{p-i}(M)L_{-i}+u_p(M, Q),
		\end{equation}
		where 
		\begin{eqnarray}\label{T-exp-ki}
			g_{1}(M)=k_1M_{-1},
			g_{i}(M)=k_iM_{-i}+k_i\sum_{j=1}^{i-1}\left(1-\frac{j}{2p-i}\right)g_{j}(M)M_{-(i-j)}, i=2, \cdots, p-1.
		\end{eqnarray}  and 
		\begin{eqnarray*}\label{T-exp-u_p}
			u_p(M, Q)&=&\sum_{\nu\in\mathcal P(p), \ell(\mu)\ge 2} d_\mu M_{-\mu}
			+\sum_{i=1}^{\lfloor \frac{p}{2}\rfloor }d_iQ_{i-p-\frac{1}{2}}Q_{-i+\frac{1}{2}}  +\sum_{\stackrel{\frac p2\ne l_1>l_2\ge 1}{\mu\in\mathcal P(p-l_1-l_2+1)}}d_{\mu}^{l_1, l_2}Q_{-l_1+\frac{1}{2}}Q_{-l_2+\frac12}M_{-\mu}
		\end{eqnarray*} with unique coefficients $d_\mu, d_{\mu}^{l_1, l_2}, d_i\in\mathbb C$.
	\end{cor}
	\begin{proof}
		For $i=p-1, p-2, \cdots, 1$, using \eqref{subsingular2} we deduce that
		$$ 0=  M_{p-i}{\rm T}\1=[M_{p-i},L_{-p}]\1+\sum_{j=1}^{p-1}g_{p-j}(M)[M_{p-i},L_{-j}]\1 =[M_{p-i},L_{-p}]\1+\sum_{j=p-i}^{p-1}g_{p-j}(M)[M_{p-i},L_{-j}]\1\\
		$$
		$$\aligned=&(2p-i)M_{-i}\1+(2p-i-1)g_1(M)M_{-i-1}\1+\cdots+
		g_{i}(M)\left( 2(p-i)M_0+\frac{(p-i)^3-(p-i)}{12}c_M\right)\1.
		\endaligned$$
		Applying $h_M=-\frac{p^2-1}{24}c_M$ we deduce that 
		$$(2p-i)M_{-i} +(2p-i-1)g_1(M)M_{-i-1} +\cdots+
		g_{i}(M)\left( 2(p-i)M_0+\frac{(p-i)^3-(p-i)}{12}c_M\right)=0$$
		\begin{eqnarray*}\label{giw}
			g_{1}(M)=k_1M_{-1},
			g_{i}(M)=k_iM_{-i}+k_i\sum_{j=1}^{i-1}\left(1-\frac{j}{2p-i}\right)g_{j}(M)M_{-(i-j)}, i=2, \cdots, p-1.
		\end{eqnarray*}
		So \eqref{T-exp-ki} follows by induction.
		
		By actions of $Q_{i-\frac12}, i=p, p-1, \cdots, 1$ on \eqref{T-exp}, we can get all $d_i$ by induction.
		Meanwhile, by actions of $L_i, i=p-1, \cdots, 1$ on \eqref{T-exp}, we can get all $d_{\mu}^{l_1, l_2}, d_\mu$ by induction.
	\end{proof}

	\begin{exa}
		(1) $p=4, h_M=-\frac{5}{8}c_M, h_L=-\frac{5}{8}c_L+\frac{153}{16}:
				$
			\begin{eqnarray*}{\rm T}{=}&L_{-4}+\frac{4}{c_M}M_{-1}L_{-3}+\left(\frac{3}{c_M}M_{-2}+\frac{10}{c_M^{2}}M_{-1}^{2}\right)L_{-2}
			+\left(\frac{4}{c_M}M_{-3}+\frac{20}{c_M^2}M_{-2}M_{-1}+\frac{24}{c_M^3}M_{-1}^{3}\right)L_{-1}\\
			&-\frac{2}{c_M}Q_{-\frac{7}{2}}Q_{-\frac{1}{2}}-\frac{6}{c_M}Q_{-\frac{5}{2}}Q_{-\frac{3}{2}}
			-\frac{16}{c_M^2}M_{-1}Q_{-\frac{5}{2}}Q_{-\frac{1}{2}}+\frac{6}{c_M^2}M_{-2}Q_{-\frac{3}{2}}Q_{-\frac{1}{2}}
			-\frac{12}{c_M^3}M_{-1}^2Q_{-\frac{3}{2}}Q_{-\frac{1}{2}}\\
			&+\left(\frac{66}{c_M^2}-\frac{4c_L}{c_M^2}\right)M_{-3}M_{-1}+\left(\frac{51}{4c_M^2}-\frac{3c_L}{2c_M^2}\right)M_{-2}^2
			+\left(\frac{342}{c_M^3}-\frac{20c_L}{c_M^3}\right)M_{-2}M_{-1}^2+\left(\frac{321}{c_M^4}-\frac{18c_L}{c_M^4}\right)M_{-1}^4.
		\end{eqnarray*}
		
		{\small  (2) $p=5, h_M=-c_M, h_L=-c_L+\frac{35}{2}$: 
			\begin{eqnarray*}
				{\rm T}\hskip -7pt&=&\hskip -7pt L_{-5}+\frac{3}{c_M}M_{-1}L_{-4}+\left(\frac{2}{c_M}M_{-2}+\frac{21}{4c_M^{2}}M_{-1}^{2}\right)L_{-3}
				+\left(\frac{2}{c_M}M_{-3}+\frac{8}{c_M^2}M_{-2}M_{-1}+\frac{15}{2c_M^3}M_{-1}^3\right)L_{-2}\\
				&&+\left(\frac{3}{c_M}M_{-4}+\frac{21}{2c_M^{2}}M_{-3}M_{-1}+\frac{4}{c_M^{2}}M_{-2}^2
				+\frac{45}{2c_M^{3}}M_{-2}M_{-1}^2+\frac{45}{4c_M^{4}}M_{-1}^4\right)L_{-1}\\
				&&-\frac{3}{2c_M}Q_{-\frac{9}{2}}Q_{-\frac{1}{2}}
				-\frac{3}{c_M}Q_{-\frac{7}{2}}Q_{-\frac{3}{2}}-\frac{27}{4c_M^{2}}M_{-1}Q_{-\frac{7}{2}}Q_{-\frac{1}{2}}
				+\frac{3}{2c_M^{2}}M_{-3}Q_{-\frac{3}{2}}Q_{-\frac{1}{2}}\\
				&&-\frac{3}{c_M^{3}}M_{-2}M_{-1}Q_{-\frac{3}{2}}Q_{-\frac{1}{2}}+\frac{9}{4c_M^{4}}M_{-1}^3Q_{-\frac{3}{2}}Q_{-\frac{1}{2}}
				+\left(\frac{105}{2c_M^{2}}-\frac{3c_L}{c_M^{2}}\right)M_{-4}M_{-1}+\left(\frac{31}{c_M^{2}}-\frac{2c_L}{c_M^{2}}\right)M_{-3}M_{-2}\\
				&&+\left(\frac{369}{2c_M^{3}}-\frac{21c_L}{2c_M^{3}}\right)M_{-3}M_{-1}^2
				+\left(\frac{148}{c_M^{3}}-\frac{8c_L}{c_M^{3}}\right)M_{-2}^2M_{-1}
				+\left(\frac{1653}{4c_M^{4}}-\frac{45c_L}{2c_M^{4}}\right)M_{-2}M_{-1}^3
				+\left(\frac{675}{4c_M^{5}}-\frac{9c_L}{c_M^{5}}\right)M_{-1}^5.
			\end{eqnarray*}
		}
	\end{exa}
	Note that we have the particular element ${\rm T}\in U(\frak{g})$ but we will use ${\rm T}$ without assuming the condition that $h_L=h_{p, 1}$.
	Now we provide some key properties of the operators ${\rm S}, {\rm R}$ and ${\rm T}$ in $L'(c_L,c_M,h_L,h_M) $ without assuming that $h_L=h_{p, 1}$.
	
	\begin{lem}\label{ST}
		Let $p$ be even and ${\rm S}, {\rm T}$ be defined as above. In $L'(c_L,c_M,h_L,h_M) $, we have that $[{\rm S},{\rm T}]\1=0$, and consequently, ${\rm S}{\rm T}^i\1=0$ for any $i\in\mathbb Z_+$.  
	\end{lem}
	\begin{proof}  Note that $p>1$.
		We claim that if $[{\rm S}, {\rm T}]\1\ne 0$, then $[{\rm S}, {\rm T}]\1$ is a subsingular vector in $V(c_L,c_M, h_L,h_M)$ in the case of $h_L=h_{p, 1}$. In fact, using $[M_0,[{\rm S}, {\rm T}]]\1=0$ and \eqref{W0T} it is easy to see $[{\rm S},{\rm T}]\1$ is a $\mathfrak{g}_0$ eigenvector.
		For any $x\in\frak g_+$, 
		$$x[{\rm S},{\rm T}]\1=x{\rm S}{\rm T}\1  =[x, {\rm S}]{\rm T}\1,  \text{  in  } L'(c_L,c_M, h_{p, 1},h_M).$$
		By Lemma \ref{l3.15'}, we get $[x,{\rm S}]{\rm T}\1=0$. So the claim holds. However,$[{\rm S},{\rm T}]\1$ is not a subsingular vector in $V(c_L,c_M, h_{p, 1},h_M)_{2p}$ by ${\rm hm}([{\rm S},{\rm T}]\1)\neq L_{-p}^{2}{\bf 1}$ and Lemma \ref{hmsubsingular}.  So $[{\rm S}, {\rm T}]\1=0$. It means that 
		$[{\rm S}, {\rm T}]= y{\rm S}$ for some $y\in U(\frak g_-)$ since $p$ is even. So ${\rm S}{\rm T}\1=0$ in $L'(c_L,c_M,h_L,h_M)$ for arbitrary $h_L$.  Moreover, $${\rm S}{\rm T}^2\1=[{\rm S},{\rm T}]{\rm T}\1+{\rm T}{\rm S}{\rm T}\1=y{\rm S}{\rm T}\1+{\rm T}{\rm S}{\rm T}\1=0.$$ 
		By induction we can get ${\rm S}{\rm T}^i\1=0$ for any $i\in\mathbb Z_+$.
	\end{proof}
	
	\begin{lem}\label{RTcomm} If $p$ is odd, in $L'(c_L,c_M,h_L,h_M) $, we have	$[{\rm R}, {\rm T}]\1=0$, and   ${\rm R}{\rm T}^i\1=0$ for any $i\in\mathbb Z_+$.
		Consequently,	${\rm S}{\rm T}^i\1=0$ for any $i\in\mathbb Z_+$.
	\end{lem}
	\begin{proof}
		It is essentially the same as that of  Lemma \ref{ST}, the only difference is that we shall use Lemma \ref{l3.15} here instead of Lemma \ref{l3.15'}.
		%
	\end{proof}
	%

	\subsection{Sufficient condition for the existence of subsingular vectors}
	For any $k\in\mathbb Z_+$, set
	\begin{equation}\mathcal M_-(p)={\rm span}_{\mathbb C}\{M_{-1}, M_{-2}, \cdots, M_{-p+1}\}\label{M_-(p)},\end{equation} 
	\begin{equation} U^{(k)}:={\rm span}_{\mathbb C}\Big \{x_{i_1}x_{i_2}\cdots x_{i_{k}}\mid x_{i_1}, x_{i_2}, \cdots,  x_{i_{k}}\in U(\mathcal M_-(p))\cup \{{\rm T}\} \Big\},\label{UTk}\end{equation} (each monomial can only have a maximum $k$ copies of  ${\rm T}$) and 
	$U^{(0)}=\mathbb C$.
	
	Clearly, $$U^{(0)}\subset U^{(1)}\subset \cdots\subset U^{(k)}\subset\cdots.$$

	First we   give the following lemmas by direct calculation to show the existence of singular vectors in $L'(c_L,c_M,h_L,h_M)_{rp}$.
	
	\begin{lem} \label{g+T}
		{\rm (a)}	For any $1\le i\le p$, we have
		\begin{eqnarray*}
			[L_i,{\rm T}]&=&a_0(M)\beta(L_0,{\bf c}_L)+b_0\left(2M_0+\frac{p^2-1}{12}{\bf c}_M\right)+ c_0{\rm R}\\
			&&+\sum_{i=1}^{i-1}a_i(M)L_{i}+\sum_{i=1}^{p-1}b_iM_{i}+\sum_{i=1}^{\lfloor \frac{p}{2}\rfloor}c_iQ_{-i+\frac12},
		\end{eqnarray*}
		where $b_i\in U(\frak g_-)$, $c_i\in U(\mathcal M_-+\mathcal Q_-)$, $\beta(h_{p,1},c_L)=0$ and all $a_i(M)\in U(\mathcal M_-(p))$. Moreover, $[L_i,{\rm T}]\1\in U(\mathcal M_-(p))\1$ in $L'(c_L,c_M,h_L,h_M)$.	
		
		{\rm (b)}	For any $x\in \mathcal M_+$, we have
		\begin{eqnarray*}
			[x,{\rm T}]\subset {U({\mathfrak{g}}_{-})}\left(2M_0+\frac{p^2-1}{12}{\bf c}_M\right)+{U({\mathcal M}_{-})}({\mathcal M_+}).
		\end{eqnarray*} Moreover, $[x, {\rm T}]\1=0$ in $L'(c_L,c_M,h_L,h_M)$.
		
		{\rm (c)}	For any $x\in \mathcal Q_+$, we have
		\begin{eqnarray*}
			[x,{\rm T}]\subset {U({\mathfrak{g}}_{-})}\left(2M_0+\frac{p^2-1}{12}{\bf c}_M\right)+ U(\mathcal{M}_-+\mathcal Q_-){\rm R}+{U({\mathfrak{g}}_{-})}({\mathcal M_++\mathcal Q_+}).
		\end{eqnarray*} Moreover, $[x, {\rm T}]\1=0$ in $L'(c_L,c_M,h_L,h_M)$.
	\end{lem}
	\proof (a) We know that $[L_i,{\rm T}]\1=L_i{\rm T}\1=0$   in $L' (c_L,c_M, h_L,h_M)$ in the case of $h_L=h_{p, 1}$. Then the formula follows.
	
	The proofs for (b) and (c) are similar to that of (a). \qed

	\begin{lem} \label{W0Tk} For any  $k\in\mathbb Z_+$,  in $L'(c_L,c_M,h_L,h_M) $  with $\phi(p)=0$ we have
		
		{\rm (a)}	$M_{0}{\rm T}^{k}{\bf 1}=h_{M}{\rm T}^{k}{\bf 1}$, and $\left(M_{0}-\frac1{24}(p^2-1){\bf c}_M\right){\rm T}^{k}{\bf 1}=0$.
		
		{\rm (b)}  For $y={\rm S}, {\rm R}, M_i$ or $Q_{-i+\frac12}$ with $k, i\in\mathbb Z_+$,  we have   $yU^{(k)}\1=0$,
		where $U^{(k)}$ is defined in \eqref{UTk}.
		
	\end{lem}
	
	\proof (a) By \eqref{W0T}, we know that
	$M_{0}{\rm T}{\bf 1}=h_M{\rm T}{\bf 1}$ in $L'(c_L,c_M,h_L,h_M)$.
	By induction on $k$ and Lemma \ref{ST}, we get
	$$M_{0}{\rm T}^{k}{\bf 1}=[M_{0},{\rm T}]{\rm T}^{k-1}{\bf 1}+{\rm T}M_{0}{\rm T}^{k-1}{\bf 1}=p{\rm S}{\rm T}^{k-1}{\bf 1}+h_M{\rm T}^{k}{\bf 1}=h_M{\rm T}^{k}{\bf 1}.$$
	The rest of (a) are clear.
	
	(b) In the proof of  Lemma \ref{ST}, we see that ${\rm R}{\rm T}, {\rm S}{\rm T}\in
	U(\frak g_-){\rm R}+U(\frak g_-){\rm S}$.
	Using these we can deduce that ${\rm R}U^{(k)}\1={\rm S}U^{(k)}\1=0$.
	
	By Lemma \ref{g+T} (b) we have
	$M_i{\rm T}\1=0$ and $Q_{i-\frac12} {\rm T}\1=0$
	in $L'(c_L,c_M,h_L,h_M) $  (not assuming $h_L=h_{p,1}$). Consequently, $M_if_1{\rm T}f_2\1=Q_{i-\frac12}f_1{\rm T}f_2\1=0$ for any $f_1,f_2\in U(\mathcal{M}_-)$.
	The statements follow by induction on $k\in\mathbb{Z}_+$.
	\qed

	\begin{lem} \label{L0Tk}
		Let  $k\in \mathbb N$. In $L'(c_L,c_M,h_L,h_M) $  with $\phi(p)=0$ we have
		
		{\rm (a)}	$L_{0}{\rm T}^{k}{\bf 1}=(h_{L}+kp){\rm T}^{k}{\bf 1}$.
		
		{\rm (b)}  For any $L_i, i\in\mathbb Z_+$, we have $L_i{\rm T}^{k+1}\1\in U^{(k)}\1.$

	\end{lem}
	
	\begin{proof} 
		(a) follows from the fact that $[L_0, {\rm T}]=p{\rm T}$.

		(b) follows from Lemma \ref{g+T} and induction on $k$.
	\end{proof}

	\begin{lem}\label{LpT} 
		{\rm (a)} In $U(\frak g)$, we have 
		$$
		[L_{p},{\rm T}] =\alpha(L_0, M_0, {\bf c}_L, {\bf c}_M)+\sum_{i=1}^{p-1}a_i(M)L_{p-i}
		+\sum_{i>0}b_iM_i+\sum_{i>0}c_iQ_{i-\frac{1}{2}},
		$$ where
		\begin{eqnarray*}  
			\alpha(L_0, M_0, {\bf c}_L, {\bf c}_M)&=&2p\left(L_0+\frac{p^2-1}{24}{\bf c}_L\right)+\sum_{i=1}^{p-1}\frac{24(2p-i)}{c_M(p-i)}\left(M_0+\frac{i^2-1}{24}{\bf c}_M\right)\\
			&&-\sum_{i=1}^{\lfloor \frac{p}{2}\rfloor }\frac{12(\frac{3}{2}p+\frac{1}{2}-i)}{c_M(p-2i+1)}\left(M_0+\frac{i^2-i}{6}{\bf c}_M\right),
		\end{eqnarray*}  and $a_i(M)\in U(\mathcal M_-(p))$, $b_i\in U(\frak g_-), c_i\in U(\mathcal M_-+\mathcal Q_-)$.

		{\rm (b)}  For $k\in\mathbb Z_+$, $$L_p{\rm T}^k\1-2kp(h_L-h_{p, k}){\rm T}^{k-1}\1\in U^{(k-2)}\1.$$

		{\rm (c)} Let $k\in\mathbb N$, then  
		$$L_pU^{(k+1)}\1\subset U^{(k)}\1.$$
		
		%
		%

	\end{lem}
	\begin{proof} (a) From (\ref{T-exp'}) we see that 
		\begin{eqnarray*} [L_{p},{\rm T}]&=& \left[L_{p},L_{-p}+\sum_{i=1}^{p-1} \frac{12}{i(p-i)c_M} M_{-p+i}L_{-i}-\sum_{i=1}^{\lfloor \frac{p}{2}\rfloor }\frac{6}{(p-2i+1)c_M}Q_{i-p-\frac{1}{2}}Q_{-i+\frac{1}{2}}\right]\\
			&&+\sum_{i=1}^{p-1}a_i(M)L_{p-i}+\sum_{i>0}b_iM_i
			+\sum_{i>0}c_iQ_{i-\frac{1}{2}}\\
			&=&2pL_0+\frac{p^3-p}{12}{\bf c}_L
			+\sum_{i=1}^{p-1}\frac{24(2p-i)}{c_M(p-i)}\left(M_0+\frac{i^2-1}{24}{\bf c}_M\right)\\ 
			&&-\sum_{i=1}^{\lfloor \frac{p}{2}\rfloor }\frac{12(\frac{3}{2}p+\frac{1}{2}-i)}{c_M(p-2i+1)}\left(M_0+\frac{i^2-i}{6}{\bf c}_M\right)\\ 
			&&+\sum_{i=1}^{p-1}a_i(M)L_{p-i}+\sum_{i>0}b_iM_i
			+\sum_{i>0}c_iQ_{i-\frac{1}{2}}
		\end{eqnarray*} for some $a_i(M)\in U(\mathcal M_-(p))$, $b_i\in U(\frak g_-), c_i\in U(\mathcal M_-+\mathcal Q_-)$.
		
		(b)  Using (a) and Lemma \ref{L0Tk} (b), (c),  we can get (b) by induction on $k$, where $\alpha(L_0, M_0, {\bf c}_L, {\bf c}_M)\1$ is calculated as \eqref{e3.401} in the proof  of Theorem \ref{necessity}.
		
		(c)  follows from (a) (b) and some direct calculations by using induction on $k$.
	\end{proof}
	
	For  any  $n,  k\in\mathbb N$,
	by Lemma \ref{LpT} (c)  and Lemma \ref{W0Tk} (b), we see that 
	\begin{cor}\label{LpUk} If $n>k\ge0$, then
		$ L_p^nU^{(k)}\1=0$. 
	\end{cor}

	\begin{lem}\label{lprtr} For $k\in\mathbb Z_+$,
		$L_{p}^{k}{\rm T}^{k}{\bf 1}=(2p)^kk!\prod_{i=1}^{k}(h_L-h_{p,i}){\bf 1}$ in $L'(c_L,c_M,h_L,h_M)$.
	\end{lem}\proof
	Using induction on $k$  we obtain this result by Lemma \ref{LpT} and Corollary \ref{LpUk}.
	\qed

	%

	Now let's give the main theorem  about subsingular vectors.
	\begin{theo}\label{main3}
		Let  $(c_L,c_M,h_L,h_M)\in\bC^4$ such that $\phi(p)=0$ for some $p\in \mathbb Z_+$ with $c_M\ne0$. Then there exists a singular vector $L'(c_L,c_M,h_L,h_M)_n$ for $n\in\frac12\mathbb Z_+$
		if and only if $n=rp\in\mathbb Z_+$ for some $r\in\mathbb Z_+$ and $h_L=h_{p,r}$. Up to a scalar multiple, the only singular vector ${\rm T}_{p, r}\1\in L(c_L,c_M,h_L,h_M)_{rp}$ can written as
		\begin{eqnarray}\label{u'pr}
			{\rm T}_{p, r}\1=\left({\rm T}^r+v_1{\rm T}^{r-1}+\cdots +v_{r-1}{\rm T}+v_r\right)\1,
		\end{eqnarray}
		where $v_i\in U(\mathcal M_-)_{ip}$ does not involve $M_{-p}$.
	\end{theo}
	\proof
	The uniqueness of the singular vector ${\rm T}_{p, r}\1\in L(c_L,c_M,h_L,h_M)_{rp}$ is guaranteed by Theorem \ref{necessity}.
	We need only to show the existence of ${\rm T}_{p, r}\1$.
	
	The case of $r=1$ follows from Theorem \ref{subsingular}.
	
	Let $r>1$. Assume that
	\begin{eqnarray}\label{sub-gen}
		{\rm T}_{p,r}={\rm T}^r+v_1{\rm T}^{r-1}+v_2{\rm T}^{r-2}+\cdots +v_{r-1}{\rm T}+v_r,
	\end{eqnarray}
	where $v_i\in U(\mathcal M_-)_{ip}$.

	We order all the possible summands of ${\rm T}_{p,r}$: 
	\begin{eqnarray}\label{sub-term}
		{\rm T}^r, M_{-(p-1)}M_{-1}{\rm T}^{r-1}, \cdots, M_{-1}^p{\rm T}^{r-1}, M_{-2p}{\rm T}^{r-2},\cdots,M_{-1}^{2p}{\rm T}^{r-2}, \cdots, M_{-rp},\cdots, M_{-1}^{rp},
	\end{eqnarray}
	where the summands above don't involve $M_{-p}$ as factors.  Note that ${\rm T}_{p, r}\1$ is a linear combination of the terms in (\ref{sub-term}).
	We will try to find a solution for the coefficients of above summands in ${\rm T}_{p, r}\1$. We only need to consider the action of ${\mathfrak{vir}}_+$.  By the PBW theorem, we  consider the corresponding  operators
	\begin{eqnarray}\label{operators}
		L_p^r, L_{p-1}L_1L_p^{r-1}, L_1^pL_p^{r-1},L_{2p}L_p^{r-2}, L_1^{2p}L_p^{r-2}, \cdots L_{rp},\cdots, L_1^{rp}.
	\end{eqnarray}
	we get the linear equations
	\begin{equation}\label{xTpr=0}
		x{\rm T}_{p, r}\1=0 \ \mbox{in}\ L'(c_L,c_M,h_L,h_M)
	\end{equation}  for all $x$ in \eqref{operators}. The coefficient matrix of this linear equations (\ref{xTpr=0}) is a lower triangular matrix, with $(1,1)$-entry $(2p)^rr!\prod_{i=1}^{r}(h_L-h_{p,i}){\bf 1}$, and all other diagonal entries non-zero. By Lemma \ref{lprtr}, we deduce that  ${\rm T}_{p, r}\1$ is the only singular vector up to a scalar multiple in $L'(c_L,c_M,h_L,h_M)$ if and only if $h_L=h_{p,r}$ for some $r\in\mathbb Z_+$.
	\qed

	\begin{exa}(cf. \cite{R}) Let $p=1,h_M=0$. Then
		\begin{eqnarray*}
			&(1)&h_L=-\frac{1}{2}: {\rm T}_{1,2}=L_{-1}^2+\frac{6}{c_M}M_{-2};\\
			&(2)&h_L=-1: {\rm T}_{1,3}=L_{-1}^3+\frac{24}{c_M}M_{-2}L_{-1}+\frac{12}{c_M}M_{-3};\\
			&(3)&h_L=-\frac{3}{2}: {\rm T}_{1,4}=L_{-1}^4+\frac{60}{c_M}M_{-2}L_{-1}^2+\frac{60}{c_M}M_{-3}L_{-1}+\frac{36}{c_M}M_{-4}+\frac{108}{c_M^2}M_{-2}^2.
		\end{eqnarray*}
	\end{exa}

	\begin{exa}
		$p=2,r=2, h_M=-\frac{1}{8}c_M, h_L=h_{2,2}=-\frac{1}{8}c_L+\frac{5}{16}:$
		\small{	\begin{eqnarray*}
				{\rm T}_{2,2}&=&L_{-2}^2+\frac{12}{c_M}M_{-1}L_{-3}+\frac{24}{c_M}M_{-1}L_{-2}L_{-1}+\frac{144}{c_M^2}M_{-1}^2L_{-1}^2-\frac{12}{c_M}M_{-3}L_{-1}\\
				&&-\frac{12}{c_M}Q_{-\frac{3}{2}}Q_{-\frac{1}{2}}L_{-2}+\left(\frac{174}{c_M^2}-\frac{12c_L}{c_M^2}\right)M_{-1}^2L_{-2}-\frac{144}{c_M^2}M_{-1}Q_{-\frac{3}{2}}Q_{-\frac{1}{2}}L_{-1}+\left(\frac{2088}{c_M^3}-\frac{144c_L}{c_M^3}\right)M_{-1}^3L_{-1}\\
				&&-\frac{3}{c_M}Q_{-\frac{7}{2}}Q_{-\frac{1}{2}}-\frac{3}{c_M}Q_{-\frac{5}{2}}Q_{-\frac{3}{2}}-\frac{72}{c_M^2}M_{-1}Q_{-\frac{5}{2}}Q_{-\frac{1}{2}}+\left(\frac{72c_L}{c_M^3}-\frac{1476}{c_M^3}\right)M_{-1}^2Q_{-\frac{3}{2}}Q_{-\frac{1}{2}}\\
				&&+\frac{12}{c_M}M_{-4}+\left(\frac{12c_L}{c_M^2}+\frac{6}{c_M^2}\right)M_{-3}M_{-1}+\left(\frac{36c_L^2}{c_M^4}-\frac{1044c_L}{c_M^4}+\frac{2385}{c_M^4}\right)M_{-1}^4.
		\end{eqnarray*}} 
		Note that
		$p=2,r=1, h_M=-\frac{1}{8}c_M, h_L=h_{2,1}=-\frac{1}{8}c_L+\frac{21}{16}:$
		\begin{eqnarray*}
			{\rm T}=L_{-2}+\frac{12}{c_M}M_{-1}L_{-1}+\left(\frac{87}{c_M^2}-\frac{6c_L}{c_M^2}\right)M_{-1}^2-\frac{6}{c_M}Q_{-\frac{3}{2}}Q_{-\frac{1}{2}}.
		\end{eqnarray*}
		By direct calculation, we get
		\begin{eqnarray*}
			{\rm T}_{2,2}=T^2+\frac{6}{c_M}M_{-4}+\frac{216}{c_M^2}M_{-3}M_{-1}-\frac{5184}{c_M^4}M_{-1}^4.
		\end{eqnarray*}
	\end{exa}
	\qed
	
	In the above arguments, starting  from Lemma \ref{ll4.1}  to Theorem \ref{main3}, by deleting  parts (or terms) involving $\mathcal Q$ we derive the following results about the subalgebra $W(2, 2)$ of $\frak g$:
	
	\begin{cor} \label{w22-sub} Let 	$(c_L,c_M,h_L,h_M)\in\bC^4$.
		The Verma module $V_{W(2,2)}(h_L, h_M, c_L, c_M)$ over $W(2,2)$ has a subsingular vector if and only if  $\phi(p)=0$ for some $p\in\mathbb Z_+$, and
		\begin{eqnarray*}
			h_L=h_{p, r}'=-\frac{p^2-1}{24}c_L+\frac{(13p+1)(p-1)}{12}+\frac{(1-r)p}{2},
		\end{eqnarray*}
		for some $r\in\mathbb Z_+$.
		
	\end{cor}
	
	\begin{rem}

		The value \( h_{p,r}' \) is obtained by omitting the final summand in equation (\ref{e3.401}). This corollary was first conjectured in \cite{R} and  further discussed in \cite{JZ} with some new ideas. 

	\end{rem}

	%
	
	\begin{cor}\label{main2-w22}  Let 	$(c_L,c_M,h_L,h_M)\in\bC^4$ such that 
		$\phi(p)=0$ for some $p\in\mathbb Z_+$, and $h_L=h_{p, r}'$ for some $r\in\mathbb Z_+$. Then
		$$u'_{p,r}=\left({\rm T}^r+v_1{\rm T}^{r-1}+\cdots +v_{r-1}{\rm T}+v_r\right)\1$$ for $v_i\in U({\mathcal{M}}_-)_{ip}$, is the unique subsingular vector of the Verma module $V_{W(2,2)}(h_L, h_M, c_L, c_M)$, up to a scalar multiple, where
		\begin{equation}\label{T-exp-W22}
			{\rm T}=L_{-p}+\sum_{i=1}^{p-1}g_{p-i}(M)L_{-i}+\sum_{\nu\in\mathcal P(p), \ell(\nu)\ge 2} d_\nu M_{-\nu},
		\end{equation} and $g_{i}(M)$ are given in \eqref{T-exp-ki}, and $d_\nu$ can be determined as in Corollary \ref{subsingular-T} by actions of $L_i, i=p-1, p-2, \cdots, 1$.
	\end{cor}

\section{Characters of irreducible highest weight modules and   composition series }

In this section, we provide the maximal submodules of $V(c_L,c_M,h_L,h_M)$ and the character formula for irreducible highest weight modules. We also derive the composition series (of infinite length) of $V(c_L,c_M,h_L,h_M)$.

Again we fix $(c_L,c_M,h_L,h_M)\in\bC^4$, and 
assume that  $\phi(p)=2h_M+\frac{p^2-1}{12}c_M=0$ for some $p\in \mathbb Z_+$ with $c_M\neq 0$. Let us first define atypical and  typical Verma module $V(c_L,c_M,h_L,h_M)$.

\begin{defi} For $c_L,c_M\in\mathbb C$, let 
	$$
	{\mathcal {AT} }(c_L,c_M)= \left\{ \left(h_{p,r}, \frac{1-p^2}{24}c_M\right) \mid p,r \in \mathbb{Z}_+ \right\},$$
	where $h_{p,r}$ is defined in (\ref{e3.37}). 
	We say the Verma module $V(c_L,c_M,h_L,h_M)$ to be  \textit{atypical} if   $(h_L,h_M)\in \mathcal {AT}(c_L, c_M)$,   otherwise to be \textit{typical} (see \cite{AR2}).  
\end{defi}

\begin{lem} \label{R-S-lemma} Let ${\rm T}_{p,r}$ be defined in Theorem \ref{main3}, then in $V(c_L,c_M,h_L,h_M)$, we have 
	\begin{eqnarray}
		M_{(r-1)p}{\rm T}_{p,r}\1=r!p^r{\rm S}\1+\delta_{r,1}h_M{\rm T}_{p,r}\1; \label{MS}
		\\
		Q_{(r-1)p+\frac{p}{2}}{\rm T}_{p,r}\1=r!p^r{\rm R}\1, \text{ \rm if $p$ is odd}.\label{QR}
	\end{eqnarray}
\end{lem}
\begin{proof}
	Let us first prove \eqref{MS}. This is clear for $r=1$ since ${\rm T}_{p, 1}={\rm T}$ and $[M_0, {\rm T}]=p{\rm S}$. Now we assume that $r>1$.
	By \eqref{u'pr}, we have $M_{(r-1)p}{\rm T}_{p,r}\1=M_{(r-1)p}{\rm T}^r\1+v_1M_{(r-1)p}{\rm T}^{r-1}\1$.
	
	%
	By Lemma \ref{g+T} (b) and by induction on $k\ge 1$ we see that $$M_{kp}{\rm T}^{k}\1=U(\mathcal M)\left(M_0+\frac1{24}(p^2-1)c_M\right)\1=0.$$
	
	By induction on $k\ge1$ and using Lemma \ref{g+T} (b),  we can prove that $M_{(k-1)p+j}{\rm T}^k\1= 0$ for any $j\in\mathbb Z_+$.

	Now by induction on $k\ge2$ we will prove that $M_{(k-1)p}{\rm T}^k\1= k!p^k{\rm S}\1$. This is clear for $k=2$ by direct computations. We compute that 
	$$\aligned M_{(k-1)p}{\rm T}^k\1=&[M_{(k-1)p}, {\rm T}]{\rm T}^{k-1}\1+{\rm T}M_{(k-1)p}{\rm T}^{k-1}\1=[M_{(k-1)p}, {\rm T}]{\rm T}^{k-1}\1\\
	=&[M_{(k-1)p}, L_{-p}]{\rm T}^{k-1}\1+\sum_{i=1}^{p-1}g_i(M)[M_{(k-1)p}, L_{-p+i}]{\rm T}^{k-1}\1\\
	=&
	(kp)M_{(k-2)p}{\rm T}^{k-1}\1+\sum_{i=1}^{p-1}g_i(M)M_{(k-2)p+i}{\rm T}^{k-1}\1\\
	=&(kp)M_{(k-2)p}{\rm T}^{k-1}\1\\
	=&k!p^k{\rm S}\1, \,\,\, (\text{induction used}).\endaligned$$ 
	So \eqref{MS} holds.
	
	Now we prove \eqref{QR}. By induction on $k\in\mathbb Z_+$  and using Lemma \ref{g+T} (c), we can prove that
	$Q_{(k-1)p+j+\frac{p}{2}}{\rm T}^k\1=0$ for any $j\in \mathbb Z_+$.
	
	Now by induction on $k\ge1$ we will prove that $Q_{(k-1)p+\frac{p}{2}}{\rm T}^k\1= k!p^k{\rm R}\1$. This is clear for $k=1$ by direct computations. We compute that 
	$$\aligned Q_{(k-1)p+\frac{p}{2}}{\rm T}^k\1=&[Q_{(k-1)p+\frac{p}{2}}, {\rm T}]{\rm T}^{k-1}\1+{\rm T}Q_{(k-1)p+\frac{p}{2}}{\rm T}^{k-1}\1\\
	=&[Q_{(k-1)p+\frac{p}{2}}, {\rm T}]{\rm T}^{k-1}\1\\
	=&kpQ_{(k-2)p+\frac{p}{2}}{\rm T}^{k-1}\1\\
	=& k!p^k{\rm R}\1, \,\,\, (\text{induction used}).\endaligned$$ 
	Then $Q_{(r-1)p+\frac{p}{2}}{\rm T}_{p,r}\1 =Q_{(r-1)p+\frac{p}{2}}{\rm T}^r\1=r!p^r{\rm R}\1.$ 
\end{proof}

\subsection{Maximal submodules and characters}

Now we are ready to present a couple of other main theorems in this paper. 

\begin{theo}\label{irreducibility} Let $(c_L,c_M,h_L,h_M)\in\bC^4$ such that $2h_M+\frac{p^2-1}{12}c_M=0$ for some $p\in \mathbb Z_+$ with $c_M\neq 0$ and $(h_L,h_M)\not\in \mathcal{AT}(c_L, c_M)$ (typical case). Then
	$J(c_L,c_M,h_L,h_M)$,  the maximal submodule of $V(c_L,c_M,h_L,h_M)$, is generated by $  {\rm S}\1 $ if $  p\in 2\mathbb Z_+$, by $ {\rm R}\1 $ if   $p\in 2\mathbb Z_+-1 $, and 
	the simple quotient  $L(c_L,c_M,h_L,h_M)=V(c_L,c_M,h_L,h_M)/J(c_L,c_M,h_L,h_M)$ has a basis	${\mathcal B}$ in (\ref{e4.1}) if $p\in 2\mathbb Z_+$; or the   basis	${\mathcal B}'$ in (\ref{e4.2}) if $p\in 2\mathbb Z_+-1$.
	Moreover,  
	$$
	{\rm char}\, L(c_L,c_M,h_L,h_M)= q^{h_L}(1-q^{\frac{p}2})\left(1+\frac12(1+(-1)^p)q^{\frac p2}\right)\prod_{k=1}^{\infty}\frac{1+q^{k-\frac{1}{2}}}{(1-q^{k})^{2}}.		
	$$
\end{theo}

\begin{proof}
	It follows from  
	Theorems \ref{necessity} and \ref{main2}.
\end{proof}

\begin{theo}\label{irreducibility1}
 Let $(c_L,c_M,h_L,h_M)\in\bC^4$ such that $2h_M+\frac{p^2-1}{12}c_M=0$ for some $p\in \mathbb Z_+$ with $c_M\neq 0$ and $(h_L,h_M)\in \mathcal{AT}(c_L, c_M)$ (atypical case). Then
	$J(c_L,c_M,h_L,h_M)=\langle {\rm T}_{p,r}\1\rangle$ is generated by $ {\rm T}_{p,r}\1$ defined in Section 4, and $L(c_L,c_M,h_L,h_M)=V(c_L,c_M,h_L,h_M)/J(c_L,c_M,h_L,h_M)$ has  a basis
	\begin{equation}\label{B''}
		{\mathcal B_1}=\{M_{-\la}Q_{-\mu+\frac{1}{2}}L_{-\nu}{\bf 1}\mid    \la,\nu\in \mathcal P, \mu\in\mathcal{SP} \ \mbox{not involving}\ M_{-p}, L_{-p}^n (n\geq r) \}
	\end{equation} 
	if $p\in 2\mathbb Z_+$; or
	\begin{equation}\label{B''odd}
		{\mathcal B}_1'=\{M_{-\la}Q_{-\mu+\frac{1}{2}}L_{-\nu}{\bf 1}\mid   \la,\nu\in \mathcal P, \mu\in\mathcal{SP}  \ \mbox{not involving}\ M_{-p}, Q_{-\frac{p}{2}}, L_{-p}^n (n\geq r)\}
	\end{equation}
	if $ p\in 2\mathbb Z_+-1$.
	Moreover,  
	$$
	{\rm char}\, L(c_L,c_M,h_L,h_M)= q^{h_{p,r}}(1-q^{\frac{p}2})\left(1+\frac12(1+(-1)^p)q^{\frac p2}\right)(1-q^{rp})\prod_{k=1}^{\infty}\frac{1+q^{k-\frac{1}{2}}}{(1-q^{k})^{2}}.
	$$
\end{theo}

\proof 
We first consider the case that  $p$ is even. By Theorem \ref{main3}, we know that ${\rm T}_{p,r}\1$ is a subsingular vector in $V(c_L,c_M,h_{p,r},h_M)$. By Lemma \ref{R-S-lemma} we have $\langle {\rm S}\1 \rangle\subset \langle {\rm T}_{p,r}\1\rangle$. Then
\begin{eqnarray*}
	{\rm char}\, V(c_L,c_M,h_{p,r},h_M)/\langle {\rm T}_{p,r}\1\rangle =q^{h_{p,r}}(1-q^{p})(1-q^{rp})\prod_{k=1}^{\infty}\frac{1+q^{k-\frac{1}{2}}}{(1-q^{k})^{2}}.
\end{eqnarray*}

Note that ${\rm T}_{p,r}\1=(L_{-p}^r+g_1L_{-p}^{r-1}+\cdots+g_{r-1}L_{-p}+g_{r})\mathbf 1$, i.e., 
$\label{rtor-1}
L_{-p}^r\1=-(g_1L_{-p}^{r-1}+\cdots+g_{r-1}L_{-p}+g_{r})\mathbf 1
$
in $V(c_L,c_M,h_{p,r},h_M)/\langle {\rm T}_{p,r}\1 \rangle$. Then we know that $\mathcal B_1$ is a basis for  $V(c_L,c_M,h_L,h_M)/\langle {\rm T}_{p,r} \1\rangle$.

Now we determine the irreducibility of the quotient module $V(c_L,c_M,h_{p,r},h_M)/\langle {\rm T}_{p,r}\1 \rangle$. Suppose that $x\in V(c_L,c_M,h_{p,r},h_M)/\langle {\rm T}_{p,r}\1 \rangle$ is a weight vector with $h_L=h_{p,r}$ such that $U(\mathfrak{g_+})x\subset \langle {\rm T}_{p,r}\1 \rangle$. Analogous to the proof of Lemma \ref{hmsubsingular}, one can prove that ${\rm hm}(x)=L_{-p}^s\1$ for some $s<r$.  Similar arguments as the proof of Theorem \ref{necessity}, one can show that $h_{p,r}=h_{p,s}$, which is a contradiction. So $J(c_L,c_M,h_L,h_M)=\langle {\rm T}_{p,r}\1\rangle$.

Next, we consider the case that   $p$ is odd. It is clear that 
\begin{eqnarray*}
	{\rm char}\, V(c_L,c_M,h_{p,r},h_M)/\langle {\rm T}_{p,r}\1\rangle=q^{h_{p,r}}(1-q^{\frac{p}{2}})(1-q^{rp})\prod_{k=1}^{\infty}\frac{1+q^{k-\frac{1}{2}}}{(1-q^{k})^{2}}.
\end{eqnarray*}

Similar as for the  case of even $p$, one can determine the irreducibility of the quotient module $V(c_L,c_M,h_{p,r},h_M)/\langle {\rm T}_{p,r}\1 \rangle$.
So $J(c_L,c_M,h_L,h_M)=\langle {\rm T}_{p,r}\1\rangle$ in this case also.
\qed


\begin{rem} 
	From the above two theorems, unlike the Virasoro algebra, any maximal submodule of the Verma  module $V(c_L,c_M,h_{p,r},h_M)$ is always generated by a single weight vector.
\end{rem}

\subsection {Composition series of submodules}
In this subsection, we investigate the composition series of submodules of the Verma module $V(c_L,c_M,h_L,h_M)$. 
We have seen that the chain length of the submodules is infinite in the reducible Verma module $V(c_L,c_M,h_L,h_M)$.
Now we consider both typical and atypical   cases.

\begin{theo}\label{main4-1}  Let $(c_L,c_M,h_L,h_M)\in\bC^4$ such that $2h_M+\frac{p^2-1}{12}c_M=0$ for some $p\in \mathbb Z_+$ with $c_M\neq 0$ and $(h_L,h_M)\notin \mathcal{AT}(c_L, c_M)$.
	
	$(1)$ If $p\in 2\mathbb Z_+$, then the Verma module $V(c_L,c_M,h_L,h_M)$ has the following infinite composition series of submodules
	\begin{eqnarray}\label{filtration3}
		V(c_L,c_M,h_L,h_M)\supset\langle {\rm S}\1 \rangle \supset \langle {\rm S}^2\1 \rangle\supset\cdots\supset \langle {\rm S}^n\1 \rangle\supset \cdots
	\end{eqnarray}
	Moreover, $\langle {\rm S}^n\1 \rangle/\langle {\rm S}^{n+1}\1 \rangle$ is isomorphic to the irreducible highest weight module $L(c_L,c_M,h_L+np,h_M)$ for any $n\in \mathbb{N}$.\par
	$(2)$ If   $p\in 2\mathbb Z_+-1$, then the Verma module $V(c_L,c_M,h_L,h_M)$ has  the following composition  series of submodules
	\begin{eqnarray}\label{filtration4}
		V(c_L,c_M,h_L,h_M)\supset \langle {\rm R}\1 \rangle\supset\langle {\rm R}^2\1 \rangle \supset\cdots \supset \langle {\rm R}^n\1 \rangle\supset\cdots
	\end{eqnarray}
	Moreover, $\langle {\rm R}^{2n}\1 \rangle/\langle {\rm R}^{2n+1}\1 \rangle$ is isomorphic to the irreducible highest weight module $L(c_L,c_M,h_L+np,h_M)$ and  $\langle {\rm R}^{2n+1}\1 \rangle/\langle {\rm R}^{2n+2}\1 \rangle$ is isomorphic to the irreducible highest weight module $L(c_L,c_M,h_L+np+\frac{p}{2},h_M)$ for any $n\in \mathbb{N}$.
\end{theo}

\proof  
(1) We know that $\langle {\rm S}^n\1 \rangle$ is a submodule of $V(c_L,c_M,h_L,h_M)$ for any $n\in \mathbb{Z}_+$ and (\ref{filtration3}) follows from Theorem \ref{main1}. Note that $M_{0}{\rm S}^{n}\1=h_M {\rm S}^{n}\1$,  $L_{0}{\rm S}^{n}\1=(h_L+np) {\rm S}^{n}\1$. Clearly, the submodule $\langle {\rm S}^n\1 \rangle$ is isomorphic to $V(c_L,c_M,h_L+np,h_M)$ with the highest weight $(c_L,c_M,h_L+np,h_M)$.
Note that ${\rm S}^{n}\1+\langle {\rm S}^{n+1}\1 \rangle$ generates $\langle {\rm S}^{n}\1 \rangle/\langle {\rm S}^{n+1}\1 \rangle$ and $M_{0}({\rm S}^{n}\1+\langle {\rm S}^{n+1}\1 \rangle)\in h_M ({\rm S}^{n}\1+\langle {\rm S}^{n+1}\1 \rangle)$,  $L_{0}({\rm S}^{n}\1+\langle {\rm S}^{n+1}\1 \rangle)\in (h_L+np) ({\rm S}^{n}\1+\langle {\rm S}^{n+1}\1 \rangle)$. Then $\langle {\rm S}^{n}\1 \rangle/\langle {\rm S}^{n+1}\1 \rangle$ is isomorphic to a highest weight module with weight $(c_L,c_M,h_L+np,h_M)$ and highest weight vector ${\rm S}^{n}\1+\langle {\rm S}^{n+1}\1 \rangle$. By the calculation of character, we have 
\begin{eqnarray*}
	{\rm char}\, \langle {\rm S}^{n}\1 \rangle/\langle {\rm S}^{n+1}\1 \rangle    &=&q^{h_L+np}\prod_{k=1}^{\infty}\frac{1+q^{k-\frac{1}{2}}}{(1-q^{k})^{2}}-q^{h_L+(n+1)p)}\prod_{k=1}^{\infty}\frac{1+q^{k-\frac{1}{2}}}{(1-q^{k})^{2}}\\
	&=&q^{h_L+np}(1-q^p)\prod_{k=1}^{\infty}\frac{1+q^{k-\frac{1}{2}}}{(1-q^{k})^{2}}\\
	&=&{\rm char}\, L(c_L,c_M,h_L+np,h_M).
\end{eqnarray*} 
By Theorem \ref{irreducibility}, $\langle {\rm S}^{n+1}\1 \rangle$ is the maximal submodule of $\langle {\rm S}^n\1\rangle$ and the quotient module $\langle {\rm S}^n\1 \rangle/\langle {\rm S}^{n+1}\1 \rangle$ is isomorphic to $L(c_L,c_M,h_L+np,h_M)$ for $n\in \mathbb{N}$.

For Statement (2),   ${\rm R}^n{\bf 1}$ is a singular vectors in $V(c_L,c_M,h_L,h_M)$ for any $n\in\mathbb{Z}_+$ by Theorems \ref{main1}, \ref{main2}.  The submodule $\langle {\rm R}^n\1 \rangle$ generated by ${\rm R}^n\1$ is isomorphic to $V(c_L,c_M,h_L+\frac{np}{2},h_M)$. Clearly, $(\ref{filtration4})$ holds.  
Since $h_L\neq h_{p,r}$ for all $p\in 2\mathbb Z_+-1, r\in \mathbb Z_+$, the quotient $V(c_L,c_M,h_L,h_M)/\langle {\rm R}\1 \rangle=L(c_L,c_M,h_L,h_M)$ is irreducible by Theorem \ref{irreducibility}. Note that ${\rm R}^{2n}\1+\langle {\rm R}^{2n+1}\1 \rangle$ generates $\langle {\rm R}^{2n}\1 \rangle/\langle {\rm R}^{2n+1}\1 \rangle$. Moreover, $M_{0}({\rm R}^{2n}\1+\langle {\rm R}^{2n+1}\1 \rangle)\in h_M {\rm R}^{2n}\1+\langle {\rm R}^{2n+1}\1 \rangle$ and  $L_{0}({\rm R}^{2n}\1+\langle {\rm R}^{2n+1}\1 \rangle)\in (h_L+np) {\rm R}^{2n}\1+\langle {\rm R}^{2n+1}\1 \rangle$. We have 
\begin{eqnarray*}
	{\rm char}\, \langle {\rm R}^{2n}\1 \rangle/\langle {\rm R}^{2n+1}\1 \rangle    &=&q^{h_L+np}\prod_{k=1}^{\infty}\frac{1+q^{k-\frac{1}{2}}}{(1-q^{k})^{2}}-q^{h_L+(n+\frac{1}{2})p)}\prod_{k=1}^{\infty}\frac{1+q^{k-\frac{1}{2}}}{(1-q^{k})^{2}}\\
	&=&q^{h_L+np}(1-q^{\frac{p}{2}})\prod_{k=1}^{\infty}\frac{1+q^{k-\frac{1}{2}}}{(1-q^{k})^{2}}\\
	&=&{\rm char}\, L(c_L,c_M,h_L+np,h_M).
\end{eqnarray*} 
Clearly, $h_L+np\neq h_{p,r}$ for all $p\in 2\mathbb Z_+-1, r\in \mathbb Z_+$, and then the quotient  $\langle {\rm R}^{2n}\1 \rangle/\langle {\rm R}^{2n+1}\1 \rangle\cong L(c_L,c_M,h_L+np,h_M)$ for any $n\in \mathbb{N}$ is irreducible by Theorem \ref{irreducibility}.

Note that ${\rm R}^{2n+1}\1+\langle {\rm R}^{2n+2}\1 \rangle$ generates $\langle {\rm R}^{2n+1}\1 \rangle/\langle {\rm R}^{2n+2}\1 \rangle$. Moreover, $M_{0}({\rm R}^{2n+1}\1+\langle  {\rm R}^{2n+2}\1 \rangle)\in h_M {\rm R}^{2n+1}\1+\langle {\rm R}^{2n+2}\1 \rangle$ and $L_{0}({\rm R}^{2n+1}\1+\langle  {\rm R}^{2n+2}\1 \rangle)\in \left(h_L+np+\frac{p}{2}\right) {\rm R}^{2n+1}\1+\langle {\rm R}^{2n+2}\1 \rangle$.
Similarly,
\begin{eqnarray*}
	{\rm char}\,\langle {\rm R}^{2n+1}\1 \rangle/\langle {\rm R}^{2n+2}\1 \rangle &=&q^{h_L+np+\frac{p}{2}}\prod_{k=1}^{\infty}\frac{1+q^{k-\frac{1}{2}}}{(1-q^{k})^{2}}-q^{h_L+(n+1)p)}\prod_{k=1}^{\infty}\frac{1+q^{k-\frac{1}{2}}}{(1-q^{k})^{2}}\\
	&=&q^{h_L+np+\frac{p}{2}}(1-q^{\frac{p}{2}})\prod_{k=1}^{\infty}\frac{1+q^{k-\frac{1}{2}}}{(1-q^{k})^{2}}\\
	&=&{\rm char}\, L(c_L,c_M,h_L+np+\frac{p}{2},h_M).
\end{eqnarray*}
So $\langle {\rm R}^{2n+1}\1 \rangle/\langle {\rm R}^{2n+2}\1 \rangle\cong L(c_L,c_M,h_L+np+\frac{p}{2},h_M)$. 
\qed

We next consider the atypical case. 
\begin{theo} \label{main4-2}
	 Let $(c_L,c_M,h_L,h_M)\in\bC^4$ such that $2h_M+\frac{p^2-1}{12}c_M=0$ for some $p\in \mathbb Z_+$ with $c_M\neq 0$ and $(h_L,h_M)\in \mathcal{AT}(c_L, c_M)$.  
	
	$(1)$ If $p\in 2\mathbb Z_+$, then there exists $r\in\mathbb Z_+$ satisfying the Verma module $V(c_L,c_M,h_L,h_M)$ has the following infinite composition  series of submodules
	$$	\aligned\label{filtration-aS1}
	V(c_L,c_M,&h_L,h_M)=\langle {\rm S}^0\1 \rangle\supset\langle {\rm T}_{p, r}\1 \rangle \supset\langle {\rm S}\1 \rangle \supset \langle {\rm T}_{p, r-2}({\rm S}\1) \rangle\supset\cdots \\ 
	&\supset\langle {\rm S}^{\lfloor \frac{r-1}{2}\rfloor}\1 \rangle \supset\langle {\rm T}_{p, r-2\lfloor \frac{r-1}{2}\rfloor}({\rm S}^{\lfloor \frac{r-1}{2}\rfloor}\1) \rangle\supset\langle {\rm S}^{\lfloor \frac{r-1}{2}\rfloor+1}\1 \rangle\supset\langle {\rm S}^{\lfloor \frac{r-1}{2}\rfloor+2}\1 \rangle\supset \cdots
	\endaligned$$

	$(2)$  If $p\in 2\mathbb Z_+-1$, then there exists $r\in\mathbb Z_+$ satisfying the Verma module $V(c_L,c_M,h_L,h_M)$ has the following infinite composition  series of submodules
	$$	\aligned\label{filtration-aR1}\nonumber
	V(c_L,c_M,&h_L,h_M)=\langle {\rm R}^0\1 \rangle\supset\langle {\rm T}_{p, r}\1 \rangle \supset\langle {\rm R}\1 \rangle \supset \langle {\rm T}_{p, r-1}{\rm R}\1 \rangle \supset  \langle {\rm R}^2\1 \rangle \supset \langle {\rm T}_{p, r-2}{\rm R}^2\1 \rangle\supset\cdots\\ 
	&\supset\langle {\rm R}^{r-1}\1 \rangle \supset\langle {\rm T}_{p, 1}{\rm R}^{r-1}\1 \rangle\supset\langle {\rm R}^{r}\1 \rangle\supset\langle {\rm R}^{r+1}\1 \rangle\supset \cdots
	\endaligned$$
	
\end{theo}

\begin{proof} 
	Clearly,  $\langle {\rm S}^i\1\rangle$ and $\langle {\rm R}^i\1\rangle$ are Verma modules, and
	$\langle {\rm S}^i\1\rangle\cong V(c_L,c_M,h_{p,r}+ip,h_M)$, $\langle {\rm R}^i\1\rangle\cong V(c_L,c_M,h_{p,r}+{\frac 12}ip,h_M), i\in\mathbb N$.
	
	(1) Let $p\in 2\mathbb Z_+$.  By Theorem \ref{main3}, ${\rm T}_{p,r-2i}({\rm S}^i\1)$ is the unique subsingular vector of $\langle {\rm S}^i\1\rangle\cong V(c_L,c_M,h_{p,r}+ip,h_M)$  (here ${\rm S}^i\1$ is the highest weight vector of $V(c_L,c_M,h_{p,r}+ip,h_M)$).
	By Theorem \ref{irreducibility1}, $\langle {\rm T}_{p,r-2i}({\rm S}^i\1)\rangle$ is  the maximal submodule of $\langle {\rm S}^i\1\rangle$ for any $i=0, 1, \cdots, \lfloor \frac{r-1}{2}\rfloor$. 
	
	Note that
	\begin{eqnarray}\label{hpr-relation}
		h_{p,r}+ip=h_{p,1}+\frac{(1-r)p}{2}+ip=h_{p,1}+\frac{(1-(r-2i))p}{2}=h_{p,r-2i}
	\end{eqnarray}
	for $i\in \mathbb{Z}_+$ and $r-2i>0$.
	It implies 
	\begin{equation}\langle {\rm S}^i\1\rangle/\langle  {\rm T}_{p,r-2i}({\rm S}^i\1) \rangle\cong L(c_L, c_M, h_{p, r-2i}, h_M), \ i=0,1,\cdots,\textstyle{\left\lfloor \frac{r-1}{2}\right\rfloor}.\label{subquotient00}\end{equation}
	By Lemma \ref{R-S-lemma} we have $\langle {\rm S}\1\rangle\subset \langle {\rm T}_{p, r}\1 \rangle$.  
	Then
	\begin{eqnarray*}
		&&{\rm char}\, \langle {\rm T}_{p,r}\1\rangle /\langle {\rm S}\1\rangle\\
		&=&{\rm char}\,V(c_L,c_M,h_{p,r},h_M)-{\rm char}\,L(c_L,c_M,h_{p,r},h_M)-{\rm char}\,\langle {\rm S}\1 \rangle\\
		&=&q^{h_{p,r}}\prod_{k=1}^{\infty}\frac{1+q^{k-\frac{1}{2}}}{(1-q^{k})^{2}}
		-q^{h_{p,r}}(1-q^p)(1-q^{rp})\prod_{k=1}^{\infty}\frac{1+q^{k-\frac{1}{2}}}{(1-q^{k})^{2}}-q^{h_{p,r}+p}\prod_{k=1}^{\infty}\frac{1+q^{k-\frac{1}{2}}}{(1-q^{k})^{2}}\\
		&=&q^{h_{p,r}+rp}(1-q^{p})\prod_{k=1}^{\infty}\frac{1+q^{k-\frac{1}{2}}}{(1-q^{k})^{2}}=
		{\rm char}\ L(c_L,c_M,h_{p,r}+rp,h_M),
	\end{eqnarray*}
	where we have used that  $(h_{p,r}+rp=h_{p,-r},h_M)\notin \mathcal{AT}(c_L, c_M)$.
	So $\langle {\rm T}_{p,r}\1\rangle /\langle {\rm S}\1\rangle$  is isomorphic to the irreducible highest weight module $L(c_L,c_M,h_{p,r}+rp,h_M)$.

	By Lemma \ref{R-S-lemma} we have $\langle {\rm S}({\rm S}^i\1)\rangle\subset \langle {\rm T}_{p,r-2i}({\rm S}^i\1)\rangle$. Replaced $\1$ by ${\rm S}^i\1$ as above  we can get  
	\begin{eqnarray}\langle {\rm T}_{p,r-2i}({\rm S}^i\1)\rangle/\langle {\rm S}^{i+1}\1\rangle\cong L(c_L, c_M, h_{p, r}+(r-2i)p, h_M), i=0, 1, \cdots, \textstyle{\left\lfloor \frac{r-1}{2}\right\rfloor}.\label{TquotientS}\end{eqnarray}
	So the statement (1) follows from \eqref{subquotient00} and \eqref{TquotientS}.

	(2) Let $p\in 2\mathbb Z_+-1$. Note that \eqref{hpr-relation} also holds and
	\begin{eqnarray}\label{hpr-relation1}
		h_{p,r}+ip+\frac{p}{2}=h_{p,1}+\frac{(1-(r-2i-1))p}{2}=h_{p,r-2i-1}
	\end{eqnarray}
	for $i\in \mathbb{N}$ and $r-2i-1>0$.
	By Theorem \ref{irreducibility1}, $\langle {\rm T}_{p,r-i}{\rm R}^{i}\1\rangle$ is the maximal module of $\langle {\rm R}^{i}\1\rangle$. 
	Combining with \eqref{hpr-relation} and \eqref{hpr-relation1} we get
	\begin{equation*}\langle {\rm R}^{i}\1\rangle/\langle  {\rm T}_{p,r-i}{\rm R}^i\1 \rangle\cong L(c_L, c_M, h_{p, r-i}, h_M), \ i=0,1,\cdots, r-1.\label{subquotient1}\end{equation*}

	At the same time, by Lemma \ref{R-S-lemma}, $\langle {\rm R}\1\rangle$ is a submodule of $\langle {\rm T}_{p, r}\1\rangle$. 
	By direct calculation,  we can get
	\begin{eqnarray*}
		{\rm char}\, \langle {\rm T}_{p,r}\1\rangle /\langle {\rm R}\1\rangle=q^{h_{p,r}+rp}(1-q^{\frac{p}{2}})\prod_{k=1}^{\infty}\frac{1+q^{k-\frac{1}{2}}}{(1-q^{k})^{2}}=
		{\rm char}\ L(c_L,c_M,h_{p,r}+rp,h_M),
	\end{eqnarray*}
	where we have used that $(h_{p,r}+rp=h_{p,-r},h_M)\notin \mathcal{AT}(c_L, c_M)$.
	So $\langle {\rm T}_{p,r}\1\rangle /\langle {\rm R}\1\rangle$ is isomorphic to the irreducible highest weight module $L(c_L,c_M,h_{p,r}+rp,h_M).$
	
	Replaced $\1$ by ${\rm R}^i\1$, we can get $\langle {\rm T}_{p, r-i}{\rm R}^i\1\rangle/\langle {\rm R}^{i+1}\1\rangle \cong L(c_L, c_M, h_{p, r}+(r-i)p, h_M)$ is also irreducible for all $i=0, 1, \cdots, r-1$.
	So the statement (2) holds.
	%
\end{proof}

%
%
%
%

\begin{rem}
	Theorem \ref{main4-2} demonstrates that a submodule of $V(c_L,c_M,h_L,h_M)$ is not necessarily a highest weight module (or generated by singular vectors), which markedly contrasts with the case of the Virasoro algebra \cite{FF}. Indeed, the $\frak g$-module $\langle {\rm T}_{p, r}\1\rangle$ represents an extension of the Verma module $\langle {\rm S}\1\rangle$ by the irreducible highest weight module $L(c_L, c_M, h_{p, r}+rp, h_M)$, and cannot be generatable by singular vectors.
\end{rem}

In the proofs of Theorems \ref{main4-1} and \ref{main4-2}, by deleting  part (or terms) involving $\mathcal Q$ we derive the following results about the subalgebra $W(2, 2)$ of $\frak g$. 

Set $
{\mathcal {AT}}_{W(2,2)}(c_L,c_M)= \left\{ \left(h_{p,r}', \frac{1-p^2}{24}c_M\right) :\, p,r \in \mathbb{Z}_+ \right\}$,
where $h_{p,r}'$ is defined in Corollary \ref{w22-sub}. 

\begin{cor}[\cite{R}] \label{main3-w22}  Let $(c_L,c_M,h_L,h_M)\in\bC^4$ such that $2h_M+\frac{p^2-1}{12}c_M=0$ for some $p\in \mathbb Z_+$ with $c_M\neq 0$ and
  $(h_L,h_M)\not\in {\mathcal {AT}}_{W(2,2)}(c_L,c_M)$. Then the Verma module $V_{W(2,2)}(c_L,c_M,h_L,h_M)$ has the following infinite composition  series of submodules
	\begin{eqnarray*}
		V_{W(2,2)}(c_L,c_M,h_L,h_M)\supset\langle {\rm S}\1 \rangle \supset \langle {\rm S}^2\1 \rangle\supset\cdots\supset \langle {\rm S}^n\1 \rangle\supset \cdots
	\end{eqnarray*}
\end{cor}

\begin{cor}
	 Let $(c_L,c_M,h_L,h_M)\in\bC^4$ such that $2h_M+\frac{p^2-1}{12}c_M=0$ for some $p\in \mathbb Z_+$ with $c_M\neq 0$ and $(h_L,h_M)\in {\mathcal {AT}}_{W(2,2)}(c_L,c_M)$ and $r\in \mathbb{Z}_+$. Then the Verma module $V_{W(2,2)}(c_L,c_M,h_L,h_M)$ has the following infinite composition  series of submodules
	$$	\aligned
	V_{W(2,2)}(c_L,c_M,&h_L,h_M)=\langle {\rm S}^0\1 \rangle\supset\langle {\rm T}_{p, r}\1 \rangle \supset\langle {\rm S}\1 \rangle \supset \langle {\rm T}_{p, r-2}({\rm S}\1) \rangle\supset\cdots\nonumber\\ 
	&\supset\langle {\rm S}^{\lfloor \frac{r-1}{2}\rfloor}\1 \rangle \supset\langle {\rm T}_{p, r-2\lfloor \frac{r-1}{2}\rfloor}({\rm S}^{\lfloor \frac{r-1}{2}\rfloor}\1) \rangle\supset\langle {\rm S}^{\lfloor \frac{r-1}{2}\rfloor+1}\1 \rangle\supset\langle {\rm S}^{\lfloor \frac{r-1}{2}\rfloor+2}\1 \rangle\supset \cdots
	\endaligned$$
\end{cor}

\section*{Acknowledgements}

This work was supported by the National Natural Science Foundation of China (12071405)  and NSERC (311907-2020).

\bibliographystyle{amsalpha}

\end{document}